\documentclass[11pt,reqno]{amsart}
\usepackage{amsfonts}%
\usepackage{amsbsy}%
\usepackage{amssymb}%
\usepackage{amsthm}%
\usepackage{upgreek}%
\usepackage{enumitem}%
\usepackage{mathtools} 
\usepackage{mathdots}%
\usepackage{mathrsfs}%
\usepackage{microtype} 
\usepackage{stmaryrd} 
\usepackage{color}
\usepackage[usenames,dvipsnames]{xcolor}
\usepackage[pdftex,colorlinks,linkcolor=black,urlcolor=black,citecolor=black]{hyperref}%
\usepackage{fullpage}
\usepackage[T1]{fontenc} 
\usepackage{setspace}
\usepackage{scrextend}
\usepackage{microtype}

\numberwithin{equation}{section}

\newcommand{\slanted}[1]{\slshape{#1}}

\newcommand{\scaps}[1]{{\scshape #1}}
\newcommand{\bscaps}[1]{\textsc{\textbf{#1}}}

\newcommand{\mcal}{\mathcal}
\newcommand{\mbb}{\mathbb}
\newcommand{\mbs}{\boldsymbol}
\newcommand{\mand}{\; \text{and}\;}
\newcommand{\mor}{\; \text{or}\;}

\newcommand{\Ex}{\mbb{E}}
\newcommand{\Naturals}{\mbb{N}}

\newcommand{\sm}{\setminus}
\newcommand{\discup}{\dot{\cup}}
\newcommand{\floor}[1]{\left\lfloor #1 \right\rfloor}
\newcommand{\ciel}[1]{\left\lceil #1 \right\rceil}

\let\eps=\varepsilon
\let\theta=\vartheta
\let\rho=\varrho
\let\sigma=\varsigma
\let\phi=\varphi

\def\QED{$\blacksquare$}
\def\inQED{$\square$}

\renewenvironment{proof}
{\vspace{1ex}\noindent{\slanted Proof.}\hspace{0.5em}}{\hfill \QED \vspace{1ex}}

\newenvironment{innerproof}
{\vspace{1ex}\noindent{\slanted Proof.}\hspace{0.5em}}{\hfill \inQED \vspace{1ex}}

\newenvironment{proofof}[1]
{\vspace{1ex}\noindent\scaps{Proof of #1.}\hspace{0.5em}}{\hfill \QED \vspace{1ex}}

\newenvironment{theorem}
{
\refstepcounter{equation} 
\vspace{-1ex}
\ \\
\noindent
\begin{it}
\noindent
\bscaps{Theorem~\theequation.}\hspace{-0.5ex}
}
{\end{it} \vspace{1ex}}

\newenvironment{lemma}
{
\refstepcounter{equation} 
\vspace{-1ex}
\ \\
\noindent
\begin{it}
\noindent
\bscaps{Lemma~\theequation.}\hspace{-0.5ex}
}
{\end{it} \vspace{1ex}}

\newenvironment{claim}
{
\refstepcounter{equation} 
\vspace{-1ex}
\ \\
\noindent
\begin{it}
\noindent
\bscaps{Claim~\theequation.}\hspace{-0.5ex}
}
{\end{it} \vspace{1ex}}

\newenvironment{definition}
{
\refstepcounter{equation} 
\vspace{-1ex}
\ \\
\noindent
\begin{it}
\noindent
\bscaps{Definition~\theequation.}\hspace{-0.5ex}
}
{\end{it} \vspace{1ex}}

\newenvironment{conjecture}
{
\refstepcounter{equation} 
\vspace{-1ex}
\ \\
\noindent
\begin{it}
\noindent
\bscaps{Conjecture~\theequation.}\hspace{-0.5ex}
}
{\end{it} \vspace{1ex}}

\newcommand{\A}{\mcal{A}}
\newcommand{\B}{\mcal{B}}
\newcommand{\C}{\mcal{C}}

\newcommand{\F}{\mcal{F}}
\newcommand{\G}{\mcal{G}}

\newcommand{\N}{\mcal{N}}

\renewcommand{\P}{\mcal{P}}

\newcommand{\R}{\mcal{R}}

\newcommand{\X}{\mcal{X}}

\makeatletter
\def\section{\@ifstar\unnumberedsection\numberedsection}
\def\numberedsection{\@ifnextchar[
  \numberedsectionwithtwoarguments\numberedsectionwithoneargument}
\def\unnumberedsection{\@ifnextchar[
  \unnumberedsectionwithtwoarguments\unnumberedsectionwithoneargument}
\def\numberedsectionwithoneargument#1{\numberedsectionwithtwoarguments[#1]{#1}}
\def\unnumberedsectionwithoneargument#1{\unnumberedsectionwithtwoarguments[#1]{#1}}
\def\numberedsectionwithtwoarguments[#1]#2{%
  \ifhmode\par\fi
  \removelastskip
  \vskip 1.7ex\goodbreak
  \refstepcounter{section}%
  \begingroup
  \noindent\leavevmode\large\lsstyle\scshape\normalsize
  \begin{center}\S \thesection.\ #2\end{center} 
  \endgroup
  \addcontentsline{toc}{section}{%
    \protect\numberline \thesection\ \ . \hspace{-0.3ex} {\scshape\lsstyle #1}}%
  }
\def\unnumberedsectionwithtwoarguments[#1]#2{%
  \ifhmode\par\fi
  \removelastskip
  \vskip 1.7ex\goodbreak
  \begingroup
  \noindent\leavevmode\Large\bfseries\scshape\centering 
  \begin{center} #2 \end{center} \par
  \endgroup
  \vskip 2ex\nobreak
  \addcontentsline{toc}{section}{%
    \hspace{1ex} #1}%
  }
\makeatother

\makeatletter
\def\subsection{\@ifstar\unnumberedsubsection\numberedsubsection}
\def\numberedsubsection{\@ifnextchar[
  \numberedsubsectionwithtwoarguments\numberedsubsectionwithoneargument}
\def\unnumberedsubsection{\@ifnextchar[
  \unnumberedsubsectionwithtwoarguments\unnumberedsubsectionwithoneargument}
\def\numberedsubsectionwithoneargument#1{\numberedsubsectionwithtwoarguments[#1]{#1}}
\def\unnumberedsubsectionwithoneargument#1{\unnumberedsubsectionwithtwoarguments[#1]{#1}}
\def\numberedsubsectionwithtwoarguments[#1]#2{%
  \ifhmode\par\fi
  \removelastskip
  \vskip 1.7ex\goodbreak
  \refstepcounter{subsection}%
  \noindent
  \leavevmode
  \begingroup
  \normalsize
  \noindent
  \begin{center} \thesubsection\ {\em\lsstyle #2} \end{center} 
  \endgroup
  \addcontentsline{toc}{subsection}{%
    \hspace{2ex}\protect\numberline{\thesubsection.}%
    \hspace{-7ex} {\em #1}}%
  }
\def\unnumberedsubsectionwithtwoarguments[#1]#2{%
  \ifhmode\par\fi
  \removelastskip
  \vskip 3ex\goodbreak
  \refstepcounter{subsection}%
  \noindent
  \leavevmode
  \begingroup
  \bfseries\normalsize
  \begin{center}\bscaps{#2.} \end{center}
  \endgroup
  \addcontentsline{toc}{subsection}{%
    \hspace{1ex} #1}%
  }
\makeatother

\makeatletter
\def\subsubsection{\@ifstar\unnumberedsubsubsection\numberedsubsubsection}
\def\numberedsubsubsection{\@ifnextchar[
  \numberedsubsubsectionwithtwoarguments\numberedsubsubsectionwithoneargument}
\def\unnumberedsubsubsection{\@ifnextchar[
  \unnumberedsubsubsectionwithtwoarguments\unnumberedsubsubsectionwithoneargument}
\def\numberedsubsubsectionwithoneargument#1{\numberedsubsubsectionwithtwoarguments[#1]{#1}}
\def\unnumberedsubsubsectionwithoneargument#1{\unnumberedsubsubsectionwithtwoarguments[#1]{#1}}
\def\numberedsubsubsectionwithtwoarguments[#1]#2{%
  \ifhmode\par\fi
  \removelastskip
  \vskip 1.7ex\goodbreak
  \refstepcounter{subsubsection}%
  \noindent
  \leavevmode
  \begingroup
  \normalsize
  \noindent
  \begin{center} \thesubsubsection\ {\sl\lsstyle #2} \end{center}
  \endgroup
  \addcontentsline{toc}{subsubsection}{%
    \hspace{2ex} \protect\numberline{\thesubsubsection.}%
    \hspace{-11ex} #1}%
  }
\def\unnumberedsubsubsectionwithtwoarguments[#1]#2{%
  \ifhmode\par\fi
  \removelastskip
  \vskip 3ex\goodbreak
  \noindent
  \leavevmode
  \begingroup
  \bfseries
  \bscaps{#2.}\  
  \endgroup
  \addcontentsline{toc}{subsubsection}{%
     #1}%
  }
\makeatother

\usepackage{chngcntr}
\counterwithin*{paragraph}{section}
\makeatletter
\def\namedlabel#1#2{\begingroup
    #2%
    \def\@currentlabel{#2}%
    \phantomsection\label{#1}\endgroup
}
\makeatother

\def\reg{\mathrm{reg}}

\def\matref{\mathrm{\ref{lem:matching}}}
\def\pcref{\mathrm{\ref{lem:path-cover}}}

\def\regref{\mathrm{\ref{lem:reg}}}
\def\conref{\mathrm{\ref{lem:connect}}}
\def\absref{\mathrm{\ref{lem:path-absorb}}}
\def\absrefone{\mathrm{\ref{lem:path-absorb-2}}}

\def\resrefone{\mathrm{\ref{lem:res-2}}}
\def\cntref{\mathrm{\ref{lem:cnt}}}
\def\Fref{\mathrm{\ref{lem:F}}}

\def\conrefone{\mathrm{\ref{lem:con-1deg}}}

\usepackage{tikz}
\tikzstyle{aNode} = [circle, fill = black]
\tikzstyle{bNode} = [circle,draw = black, thick]

\newcommand{\pcherry}[1]{%
\begin{tikzpicture}[inner sep = 1pt, #1]%
\node (1) at (0,-2) [aNode]{};
\node (3) at (1.5,-2) [aNode]{};
\node (2) at (0.75,-1) [aNode]{};
\draw  (1) -- (2);
\draw  (2) -- (3);
\end{tikzpicture}%
}

\newcommand{\ppoints}[1]{%
\begin{tikzpicture}[inner sep = 1pt, #1]%
\node (1) at (0,-2) [aNode]{};
\node (3) at (1.5,-2) [aNode]{};
\node (2) at (0.75,-1) [aNode]{};
\end{tikzpicture}%
}

\newcommand{\pedge}[1]{%
\begin{tikzpicture}[inner sep = 1pt, #1]%
\node (1) at (0,-2) [aNode]{};
\node (3) at (1.5,-2) [aNode]{};
\node (2) at (0.75,-1) [aNode]{};
\draw  (1) -- (3);
\end{tikzpicture}%
}

\newcommand{\ptri}[1]{%
\begin{tikzpicture}[inner sep = 1pt, #1]%
\node (1) at (0,-2) [aNode]{};
\node (3) at (1.5,-2) [aNode]{};
\node (2) at (0.75,-1) [aNode]{};
\draw  (1) -- (2);
\draw  (2) -- (3);
\draw  (1) -- (3);
\end{tikzpicture}%
}

\def\cherry{\pcherry{scale=0.13}}

\def\points{\ppoints{scale=0.13}}
\def\edge{\pedge{scale=0.14}}
\def\tri{\ptri{scale=0.14}}

\begin{document}

\title{\lsstyle Tight Hamilton cycles in cherry-quasirandom\\ $3$-uniform hypergraphs}
\author{{\em Elad Aigner-Horev and Gil Levy}}
\date{}
\maketitle

\begin{abstract}
We employ the absorbing-path method in order to prove two results regarding the emergence of tight Hamilton cycles in the so called {\em two-path} or {\em cherry}-quasirandom $3$-graphs.

Our first result asserts that for any fixed real $\alpha >0$, cherry-quasirandom $3$-graphs of sufficiently large order $n$ having  minimum $2$-degree at least $\alpha (n-2)$ have a tight Hamilton cycle. 

Our second result concerns the minimum $1$-degree sufficient for such $3$-graphs to have a tight Hamilton cycle. Roughly speaking, we prove that for every $d,\alpha >0$ satisfying $d + \alpha >1$, any sufficiently large $n$-vertex such $3$-graph $H$ of density $d$ and minimum $1$-degree at least $\alpha \binom{n-1}{2}$, has a tight Hamilton cycle.  
\end{abstract}

\begin{addmargin}[6em]{6em}
{\small
\renewcommand{\contentsname}{}
\tableofcontents
}
\end{addmargin}

\newpage

\section{Introduction}

A theorem of Dirac~\cite{Dirac} asserts that an $n$-vertex ($n\geq 3$)  graph whose minimum degree is at least $n/2$ contains a Hamilton cycle; moreover, the degree condition imposed here is best possible. A rich and extensive body of work now exists concerning the extent to which Dirac's result can be extended to uniform hypergraphs see, e.g.,~\cite{Mathias1, approx,sharp,GPW12,Mathias2,JY2,JY1,Refute,Katona,Kuhn1,Kuhn2,Kuhn3,RR14,Mathias3,RRS06,RRS08,RRS11}\footnote{The study of perfect matchings in hypergraphs is intimately related to the Hamiltonicity problem. We omit references to such results as our work here was not directly influenced by this line of research.}. Allow us  to not reproduce here the intricate development of these results as outstanding accounts of these already exist in the excellent surveys~\cite{Khun5,RR10,Zhao}.

We confine ourselves to $3$-uniform hypergraphs ($3$-graphs, hereafter). 
A $3$-graph $C$ is said to form a {\em loose cycle} if its vertices can be cyclically ordered such that each edge of $C$ captures $3$ vertices appearing consecutively in the ordering, every vertex is contained in an edge, and any two consecutive\footnote{Order of the edges inherited from the ordering of the vertices.} edges meet in precisely one vertex. We say that $C$ forms a {\em tight cycle} if there exists a cyclic ordering of its vertices such that every $3$ consecutive vertices in this ordering define an edge of $C$; this particularly implies that any two consecutive edges meet in precisely $2$ vertices. 

For a $3$-graph $H$ and two distinct vertices $u$ and $v$ of it, define 
$$
\deg_H(v) := |N_H(v)| := \Big|\Big\{\{x,y\} \in \binom{V(H)}{2}: \{x,y,v\} \in E(H)\Big\}\Big| =|\{e \in E(H): v \in e\}| 
$$
$$
\deg_H(u,v) := |N_H(u,v)|:= |\{w \in V(H): \{u,v,w\} \in E(H)\}| =|\{e \in E(H): \{u,v\} \subset e\}|.
$$
We refer to $\deg_H(v)$ as the {\em degree of $v$} (alternatively, $1$-{\em degree}) and to $\deg_H(u,v)$ as the {\em codegree of $u$ and $v$} (alternatively, $2$-{\em degree}). Set 
$$
\delta(H) := \min_{v \in V(H)}\deg_H(v)\; \mand\; \delta_2(H) := \min_{\{u,v\} \in \binom{V(H)}{2}} \deg_H(u,v). 
$$

Resolving a conjecture of~\cite{Katona}, first approximately~\cite{RRS06} and then accurately~\cite{RRS11}, the latter result asserts that a sufficiently large $n$-vertex $3$-graph $H$ satisfying $\delta_2(H) \geq \floor{n/2}$ contains a tight Hamilton cycle. A construction appearing in~\cite{Katona} demonstrates that the  codegree condition imposed here is best possible. Finding the correct threshold for $\delta(H)$ at which a $3$-graph $H$ admits a tight Hamilton cycle remained elusive for quite some time though. The problem has come to be known as the $5/9$-{\em conjecture}~\cite[Conjecture~2.18]{RR10} asserting that sufficiently large $n$-vertex $3$-graphs $H$ satisfying $\delta(H) \geq (5/9+o(1))\binom{n-1}{2}$ admit a tight Hamilton cycle. Constructions appearing in~\cite{RR10,RR14} establish that 
the degree condition appearing in this conjecture is (asymptotically) best possible. The authors of~\cite{MC} established that such $3$-graphs admit a tight cycle covering all but $o(n)$ of the vertices. Then, in a major breakthrough~\cite{break} (preceded by the deep result of~\cite{Mathias3} and around the same time as~\cite{MC}), the $5/9$-conjecture has been resolved.

An additional result relevant to our account is that of~\cite{LMM}. Presentation of the latter requires a brief overview regarding {\sl quasirandom} $3$-graphs. Launched in~\cite{CGW89,T2,T1}, the study of {\sl quasirandom graphs} has developed into a rich and vast theory, see, e.g.~\cite{KS}. While a canonical definition of quasirandom graphs was already captured in~\cite{CGW89,T2,T1}, for hypergraphs the pursuit after a definition extending~\cite{CGW89} took much longer. An elaborate account regarding the development of this pursuit can be  seen in~\cite{Me,weak,eigen,poset,Towsner} and references therein. Only recently with the work of~\cite{Towsner} has this pursuit came to an end; an alternative combinatorial approach to the functional analytic work of~\cite{Towsner} appears in~\cite{Me}.

Roughly speaking, for $k \geq 3$ each set system of $[k] = \{1,\ldots,k\}$ forming a maximal anti-chain gives rise to a notion of quasirandomness for $k$-graphs. 
In the case of interest to us, that is $k=3$, each of the maximal anti-chains
$$
\{\{1\},\{2\},\{3\}\}\; ,\; \{\{1,2\},\{3\}\},  \{\{1,2\},\{2,3\} \}, \mand \{\{1,2\},\{2,3\},\{1,3\}\}
$$
defines a notion of quasirandomness referred to as $*$-{\em quasirandomness} with $* \in \{\points, \edge,\cherry,\tri\}$, respectively (concrete definitions follow below); here these notions are arranged from left to right in {\sl increasing order of strength} so to speak.  

A solid understanding of $\points\;$-quasirandomness (i.e., the weakest notion) was attained in~\cite{weak,eigen}. More generally, we now know from~\cite{Me,Towsner} (and owing much to~\cite{poset}) that all these notions are well-separated and form a certain hierarchy with $\points\;$-quasirandomness at the "bottom" as the weakest notion (so it forms the broadest class of hypergraphs). In what follows, however, we will not be bothered with these notions of quasirandomness per se. Instead we shall consider weaker related notions. Borrowing notation from~\cite{Tetra2,Tetra}, given $d, \rho \in (0,1]$, an $n$-vertex $3$-graph $H$ is said to be $(\rho,d)_{\points}$-{\em dense} if 
	\begin{align}
	e_H(X,Y,Z)  := |\{(x,y,z) \in X \times Y \times Z: \{x,y,z\} \in E(H)\}| 
	 \geq d |X||Y||Z| - \rho n^3 \label{eq:points}
	\end{align}
	holds for every $X,Y,Z \subseteq V(H)$. If $\rho$ and $d$ exist yet are not made explicit, then we say that $H$ is $\points\;$-{\em dense}. The notion of $\points\;$-quasirandomness comes about if one imposes on $e_H(X,Y,Z)$ the upper bound corresponding to~\eqref{eq:points}.

Returning to Hamiltonicity, one encounters the following remarkable result of~\cite{LMM} stated here for $3$-graphs only.
	
\begin{theorem}\label{thm:LMM}{\em~\cite{LMM}}
For every $d, \alpha \in (0,1]$ there exist an $n_0$ and a $\rho >0$ such that the following holds whenever $n \geq n_0$ and even. Let $H$ be an $n$-vertex $(\rho,d)_{\points}$-dense $3$-graph satisfying $\delta(H) \geq \alpha \binom{n-1}{2}$. Then, $H$ admits a loose Hamilton cycle. 
\end{theorem}

\noindent
Theorem~\ref{thm:LMM} settles the issue of emergence of loose Hamilton cycles in  quasirandom $3$-graphs for any notion of quasirandomness  and any type of degree (the latter owing to~\cite[Remark~1.4]{RR10}). It asserts that all minimum degree conditions sufficient for the emergence of loose Hamilton cycles in quasirandom $3$-graphs are {\em degenerate} (i.e., any positive $\alpha$ suffices)\footnote{For hypergraphs with higher uniformity the full version of Theorem~\ref{thm:LMM} handles the emergence of the so called $1$-cycles.}. 

For tight cycles, however, a result analogous to Theorem~\ref{thm:LMM} does not exist for $\points\;$-quasirandom $3$-graphs. Indeed, ~\cite[Proposition~4]{LMM} asserts that for every $\rho >0$ and a sufficiently large $n$, an $n$-vertex $(\rho,1/8)_{\points}$-quasirandom $3$-graph $H$ exists satisfying $\delta(H) \geq (1/8 - \rho)\binom{n-1}{2}$ and having no tight Hamilton cycle. 
The constant $1/8$ here is not best possible though as the following construction demonstrates. Let $n\in \Naturals$ be sufficiently large and let $V = X \discup Y$ be a set of $n$ vertices such that $|X| = 2 n /3+1$ and $|Y|= n/3-1$ (assume $3 \mid n$). Let $G \sim G(n,p)$ be the random graph put on $V$ where each edge is put in $G$ independently at random with probability $p$; we determine $p$ below. Define $H$ to be the $3$-graph whose set of vertices is $V$ and whose set of edges consists of: 
\begin{itemize}
	\item all the sets $e \in \binom{V}{3}$ satisfying $G[e] \cong K_3$ and $e \subseteq X$ or $e \subseteq Y$ or $|e \cap X| =1$; 
	\item together with the sets $e \in \binom{V}{3}$ satisfying $2 = |e \cap X|:= |\{u,v\}|$ and $uv \notin E(G)$. 
\end{itemize}
An argument similar to the one used in~\cite[Construction~2]{RR14} asserts that $H$ has no tight Hamilton cycle. Indeed, no tight path can connect a triple contained in $X$ with a vertex of $Y$. Consequently, if $H$ were to admit a tight Hamilton cycle $C$ then $X$ must be an independent set in $C$ and $Y$ a vertex-cover of $C$. This together with the fact that $C$ is $3$-regular (with respect to $1$-degree, that is) we reach $n = e(C) \leq \sum_{y \in Y} \deg_C(y) = 3|Y| <n$; a contradiction. 
Every triple $e$ is taken into $H$ either with probability $p^3$ or $1-p$. Insisting  on $p^3 = 1-p$, so that $p = 0.6823$. Using binomial tail estimations it follows that it is highly likely that $H$ would have edge density $\approx 0.3177$, satisfy $\delta(H) \approx 0.245 n^2$, and be $\points\;$-dense. We acknowledge the discussions~\cite{private} regarding this construction. 

Replacing the degree condition seen in Theorem~\ref{thm:LMM} with a codegree condition would be insufficient in order to yield a result analogous to Theorem~\ref{thm:LMM}. Indeed, in~\cite{LMM} it is indicated that an adaption of the construction seen in~\cite[Proposition~4]{LMM} yields a $\points\;$-dense graph $H$ with $ \delta_2(H) \geq n/9$ admitting no tight Hamilton cycle.

\subsection{Our results}
If we were to "climb" up the hierarchy of notions of quasirandomness for $3$-graphs and strengthen the quasirandomness condition satisfied by the host $3$-graph would we then encounter an analogue of Theorem~\ref{thm:LMM} for tight Hamilton cycles? 
Let $d,\rho \in (0,1]$. An $n$-vertex $3$-graph $H$ is called $(\rho,d)_{\cherry}$-{\em dense} if 
	\begin{align}
	e_H(\vec G_1, \vec G_2) & := |\{(x,y,z) \in \P_2(\vec G_1,\vec G_2): \{x,y,z\} \in E(H)\}|  \geq d |\P_2(\vec G_1, \vec G_2)| - \rho n^3 \label{eq:cherry}
	\end{align}
	holds for every $\vec G_1,\vec G_2 \subseteq V(H) \times V(H)$,
	where
	$$
	\P_2(\vec G_1,\vec G_2) := \{(x,y,z) \in V(H)^3: (x,y) \in \vec G_1, (y,z) \in \vec G_2\}.
	$$ 
If $\rho$ and $d$ exist yet are not made explicit, then we say that $H$ is $\cherry\;$-{\em dense} (pronounced {\sl cherry-dense}). 

Our first main result asserts that the minimum codegree condition sufficient to imply the emergence of a tight Hamilton cycle in $\cherry\;$-dense $3$-graphs is degenerate.

\begin{theorem}\label{thm:main}
For every $d,\alpha \in (0,1]$, there exist an integer $n_0$ and a real $\rho >0$ such that the following holds for all $n \geq n_0$. Let $H$ be an $n$-vertex $(\rho,d)_{\cherry}$-dense $3$-graph satisfying $\delta_2(H) \geq \alpha (n-2)$. Then, $H$ has a tight Hamilton cycle. 
\end{theorem}

In Theorem~\ref{thm:main}, the parameter $d$ plays a somewhat docile role in the sense that no strict conditions other than it being fixed and positive need be imposed. For our second result, we consider $\cherry\;$-dense $3$-graphs with an imposed minimum $1$-degree condition. Here, a condition on $d$ arises (for us) as follows. 


\begin{theorem}\label{thm:main2}
For every $d,\alpha \in (0,1]$ satisfying $\alpha +d >1$, there exist an integer $n_0$ and a real $\rho >0$ such that the following holds for all $n \geq n_0$. Let $H$ be an $n$-vertex $(\rho,d)_{\cherry}$-dense $3$-graph satisfying $\delta(H) \geq \alpha \binom{n-1}{2}$. Then, $H$ has a tight Hamilton cycle. 
\end{theorem}

Unlike Theorem~\ref{thm:main}, the requirement $\alpha +d >1$ does not allow for a degenerate minimum $1$-degree condition. Nevertheless, it is more flexible than other results mentioned thus far. We conjecture (with some hesitation) that the condition $\alpha + d > 1$ appearing in Theorem~\ref{thm:main} can be replaced with degenerate conditions for both $\alpha$ and $d$ as follows. 

\begin{conjecture}\label{conj:main}
For every $d,\alpha \in (0,1]$ there exist an $n_0$ and $\rho >0$ such that the following holds for all $n \geq n_0$. Let $H$ be an $n$-vertex $(\rho,d)_{\cherry}$-dense $3$-graph satisfying $\delta(H) \geq \alpha \binom{n-1}{2}$. Then, $H$ has a tight Hamilton cycle. 
\end{conjecture}

"Between" $\points\;$-quasirandomness and $\cherry\;$-quasirandomness, there lies $\edge\;$-quasirandomness. For $d,\rho \in (0,1)$ an $n$-vertex $3$-graph $H$ is called $(\rho,d)_{\edge}$-{\em dense} if 
	\begin{align}
	e_H(\vec P,X) & := 
	\big|\big\{\big((u,v),x\big) \in \vec P \times X: \{u,v,x\} \in E(H)\big\}\big|  \geq d |\vec P||X| - \rho n^3 \nonumber
	\end{align}
	holds for every $\vec P \subseteq V(H) \times V(H)$ and every $X \subseteq V(H)$. 
	Unlike $\cherry\;$-quasirandom $3$-graphs, for which the Tur\'an density of $K_4^{(3)}$ (the complete $3$-graph on $4$ vertices) is zero~\cite{Tetra2}, the Tur\'an density of $K_4^{(3)-}$ (i.e., $K_4^{(3)}$ with a single edge removed) in $\edge\;$-quasirandom $3$-graphs is $1/4$~\cite{Tetra}. The {\sl absorbing configurations} (see~\S~\ref{sec:absorb} for details) used in this account involve copies of $K_4^{(3)-}$. Consequently results in the spirit of Theorems~\ref{thm:main} and~\ref{thm:main2} cannot possibly be attained for $\edge\;$-quasirandom $3$-graphs using the absorbing-path method and the absorbing configurations used in our account. We subscribe to the point of view that the flaw is not in the method and that for $\edge\;$-quasirandom $3$-graphs the minimum $1$-degree and $2$-degree conditions sufficient to imply tight Hamiltonicity are both non-degenerate. The fact that the Tur\'an density of $K_4^{(3)-}$ in $\points\;$-quasirandom $3$-graphs coincides with that seen in $\edge\;$-quasirandom $3$-graphs~\cite{Tetra} makes it not far-fetched to suspect that the minimum degree conditions in $\{\points,\edge\}$-quasirandom $3$-graphs coincide as well.  \vspace{2ex}
	
\noindent
\scaps{Open problems.} Are the following true? 
\begin{itemize}
	\item For every $d > 1/3$ and $\eps >0$, there exist an integer $n_0$ and a real $\rho>0$ such that the following holds whenever $n \geq n_0$. Let $H$ be an $n$-vertex $\points\;$-dense $3$-graph of density $d$ and satisfying $\delta_2(H) \geq n/3+\eps n$. Then, $H$ has a tight Hamiltonian cycle. 
	\item For every $d > 1/2$ and $\eps >0$, there exist an integer $n_0$ and a real $\rho>0$ such that the following holds whenever $n \geq n_0$. Let $H$ be an $n$-vertex $\points\;$-dense $3$-graph of density $d$ and satisfying $\delta_1(H) \geq n^2/4+\eps n$. Then, $H$ has a tight Hamilton cycle. 
	\item In the two questions above replace $\points\;$-denseness with $\edge\;$-denseness. 
\end{itemize}	
	
During the review and revision of this manuscript Ara\'ujo, Piga, and Schacht~\cite{AMUC} announced to have proved that for every $\eps >0$ there exists a $\rho >0$ such that every sufficiently large $(\rho,1/4+\eps)_{\edge}$-dense $3$-graph $H$ satisfying $\delta(H) \geq \eps \binom{n-1}{2}$, contains a tight Hamilton cycle. Moreover, that the constant $1/4$ is optimal. Their result implies our Theorem~\ref{thm:main2} and settles some of the questions appearing above.
At the time of writing these lines the full proof of their result was not avialable to us.

\subsection{Our approach}\label{sec:approach}
We employ the so called {\em absorbing path method} introduced in~\cite{RRS06} and further developed in~\cite{RRS08,RRS09}. Roughly speaking, this method reduces the problem of finding a tight Hamilton cycle to that of finding a tight cycle supporting two properties. First, it covers all but $\zeta n$ vertices for some carefully chosen fixed "small" $\zeta \in (0,1)$. Second, it contains a special path referred to as an {\sl absorbing-path} (rigorously defined below) which has the capability of being rerouted using only those "missing" $\zeta n$ vertices while keeping its ends unchanged and in this manner {\sl absorb}, so to speak, all missing vertices rendering a tight Hamilton cycle. Numerous reincarnations of this method now exist in the literature see e.g.,~\cite{Mathias2,LMM,break,Mathias3}. We consequently omit a more rigorous outline of this method and proceed directly to the statement of the so called {\em pillar lemmata} underlying this method; these being the so called {\em connecting lemma}, {\em absorbing-path lemma}, {\em path-cover lemma}, and {\em reservoir lemma}.

By a {\em $k$-path} we mean a $3$-graph $P$ on $k$ vertices and $k-2$ edges such that there exists a labelling of $V(P)$ namely $v_1,\ldots,v_k$ such that $\{v_i,v_{i+1},v_{i+2}\} \in E(P)$ for every $i \in [1,k-2]$. It is said that $P$ {\em connects} the pairs $\{v_1,v_2\}$ and $\{v_{k-1}, v_{k-2}\}$; also referred to as the {\em end-pairs} or simply the {\em ends} of $P$.
Throughout, the term {\em path} is used to mean a tight path. 

Roughly speaking, in the absorbing path method, the role of the {\sl connecting lemma} is, as its name suggests, to connect two disjoint pairs of vertices via a short path. A trivial precondition for such a lemma is that 
the given pairs that are to be connected both admit some non-trivial codegree.
The $3$-graphs of Theorem~\ref{thm:main} come equipped with a minimum codegree assumption which although degenerate will be sufficient in order to establish such a lemma owing to the $\cherry\;$-denseness of the host $3$-graph. The $3$-graphs of Theorem~\ref{thm:main2}, however, do not support a minimum codegree condition. As a result we will require two {\sl separate} connecting lemmas; one for each of our main results.  

Our connecting lemma fitting for Theorem~\ref{thm:main} reads as follows.  

\begin{lemma}\label{lem:connect}{\em\bf (Connecting lemma: $2$-degree)}
For every $d_{\conref}, \alpha_{\conref} \in (0,1]$, there exist an integer $n_{\conref}:= n_{\conref}(d_{\conref}, \alpha_{\conref})$ and a real $\rho_{\conref} := \rho_{\conref}(d_{\conref},\alpha_{\conref}) >0$ such that the following holds for all $n \geq n_{\conref}$ and 
$0 < \rho < \rho_{\conref}$. 

Let $H$ be an $n$-vertex $(\rho,d_{\conref})_{\cherry}$-dense $3$-graph satisfying $\delta_2(H) \geq \alpha_{\conref} (n-2)$ and let $\{x,y\}$ and $\{x',y'\}$ be two disjoint pairs of vertices. Then, there exists a $10$-path in $H$ connecting $\{x,y\}$ and $\{x',y'\}$. 
\end{lemma}

The premise of Theorem~\ref{thm:main2} allows for $3$-graphs with pairs of vertices having codegree zero or one that is too modest for our methods to work. Fortunately, regardless of any degree conditions, $\cherry\;$-dense $3$-graphs admit a certain {\sl statistical} minimum codegree condition in the sense that {\sl most} pairs of vertices admit a {\sl meaningful} codegree. This we make precise below in~\eqref{eq:Bbeta-size-2}. Unlike Lemma~\ref{lem:connect} then, the connecting lemma fitting for Theorem~\ref{thm:main2} appeals to this statistical minimum codegree condition, and upon a judicious choice of parameters connects pairs of vertices whose codegree is sufficiently high. 
It is in this lemma that we encounter the following function $g(\cdot)$. Given reals $x,y > 0$ satisfying $x + y > 1$, let 
\begin{equation}\label{eq:g-func}
g(x,y):= \min\{x,y,(x+y-1)/(y+1)\}
\end{equation}
The inequality $\alpha +d >1$ appearing in Theorem~\ref{thm:main2} traces back 
to the third term of this function. 

\begin{lemma}\label{lem:con-1deg}{\em\bf (Connecting lemma: $\mbs{1}$-degree)}
For every $d_{\conrefone}, \alpha_{\conrefone}, \eta_{\conrefone} \in (0,1]$, satisfying $\alpha_{\conrefone} + d_{\conrefone} >1$, and $\eta_{\conrefone} < g(\alpha_{\conrefone},d_{\conrefone})$, there exist an integer $n_{\conrefone}:= n_{\conrefone}(d_{\conrefone}, \alpha_{\conrefone},\eta_{\conrefone})$ and a real $\rho_{\conrefone} := \rho_{\conrefone}(d_{\conrefone},\alpha_{\conrefone},\eta_{\conrefone}) >0$ such that the following holds for all $n \geq n_{\conrefone}$ and  
$0 < \rho < \rho_{\conrefone}$.

Let $H$ be an $n$-vertex $(\rho,d_{\conrefone})_{\cherry}$-dense $3$-graph satisfying $\delta(H) \geq \alpha_{\conrefone} \binom{n-1}{2}$ and let $\{x,y\}$ and $\{x',y'\}$ be two disjoint pairs of vertices each having codegree at least $(d_{\conrefone}-\eta_{\conrefone})n$. Then, there exists a $10$-path in $H$ connecting $\{x,y\}$ and $\{x',y'\}$. 
\end{lemma}  

It is in fact true that for $\cherry\;$-dense $3$-graphs, a connecting lemma imposing no minimum degree conditions of any kind is possible. Such a lemma is presented in Lemma~\ref{lem:connect-useless} in~\S~\ref{sec:useless}. Alas, for our needs this lemma is insufficient and this too is explained in~\S~\ref{sec:useless}.

A path $A$ in an $n$-vertex $3$-graph $H$, is said to be $m$-{\em absorbing} if for every set $U \subseteq V(H) \sm V(A)$ with $|U| \leq m$ there is a path $A_U$ having the same ends as $A$ and satisfying $V(A_U) = V(A) \cup U$. A path is said to be a $(\beta,\mu,\kappa)$-{\em absorbing-path}, if it is $\mu n$-absorbing, has length at most $\kappa n$, and both its ends have codegree at least $\beta n$ in $H$.

The "split" between the aforementioned connecting lemmas propagates (for us) onwards onto the absorbing-path lemmas leading to a need to support two separate such lemmas. The absorbing-path lemma fitting for Theorem~\ref{thm:main} reads as follows.      

\begin{lemma}\label{lem:path-absorb}{\em\bf (Absorbing-path lemma: $2$-degree)}
For every  $d_{\absref},\alpha_{\absref}, \beta_{\absref} \in (0,1]$ such that $ \beta_{\absref} < \min\{d_{\absref},\alpha_{\absref}\}$, there exist an integer $n_{\absref} := n_{\absref}(d_{\absref},\alpha_{\absref}, \beta_{\absref})$, a real $\rho_{\absref} :=  \rho_{\absref}(d_{\absref},\alpha_{\absref},\beta_{\absref})>0$, a real $0<\kappa_{\absref}:=\kappa_{\absref}(d_{\absref},\alpha_{\absref},\beta_{\absref}) \leq \beta_{\absref}/2$, and a real $\mu_{\absref}:=\mu_{\absref}(d_{\absref},\alpha_{\absref}) >0$ such that 
the following holds
whenever $n \geq n_{\absref}$ and $0 < \rho < \rho_{\absref}$. 

If $H$ is an $n$-vertex $(\rho,d_{\absref})_{\cherry}$-dense $3$-graph satisfying $\delta_2(H) \geq \alpha_{\absref} (n-2)$, then it admits a $(\beta_{\absref},\mu_{\absref},\kappa_{\absref})$-absorbing-path. 
\end{lemma}

The condition $\beta_{\absref} < \min\{d_{\absref},\alpha_{\absref}\}$ appearing in the last lemma seems somewhat puzzling in view that part of the premise of the lemma is that $\delta_2(H) \geq \alpha_{\absref}(n-2)$. To a certain extent, this condition can be mitigated. We incur it here due to having certain ingredients required for the proofs of Lemma~\ref{lem:path-absorb} and its counterpart, namely Lemma~\ref{lem:path-absorb-2} stated next, consolidated. This then mandates that $\beta_{\absref} < d_{\absref}$ be imposed as to render subsequent applications of~\eqref{eq:Bbeta-size-2} meaningful. The condition $\beta_{\absref} < \alpha_{\absref}$ is admittingly "artificial"; it is kept for brevity purposes seen in the proof of Lemma~\ref{lem:path-absorb}. 

Given two reals $\alpha,\beta >0$, we write $\beta \ll \alpha$ to indicate that these can be set such that $\beta$, while fixed, can be chosen arbitrarily smaller than $\alpha$. The aforementioned counterpart of Lemma~\ref{lem:path-absorb} fitting for the setting of Theorem~\ref{thm:main2} reads as follows.

\begin{lemma}\label{lem:path-absorb-2}{\em\bf (Absorbing-path lemma: $1$-degree)}
For every $d_{\absrefone}, \alpha_{\absrefone}, \eta_{\absrefone} \in (0,1]$ satisfying $\alpha_{\absrefone} + d_{\absrefone} >1$ and $\eta_{\absrefone} < g(\alpha_{\absrefone},d_{\absrefone})$, there exist an integer $n_{\absrefone} := n_{\absrefone}(d_{\absrefone},\alpha_{\absrefone}, \eta_{\absrefone})$, a real $\rho_{\absrefone} :=  \rho_{\absrefone}(d_{\absrefone},\alpha_{\absrefone},\eta_{\absrefone})>0$, a real $0 < \kappa_{\absrefone}:=\kappa_{\absrefone}(d_{\absrefone},\alpha_{\absrefone},\eta_{\absrefone}) \ll \eta_{\absrefone} $, and a real $\mu_{\absrefone}:=\mu_{\absrefone}(d_{\absrefone},\alpha_{\absrefone}, \eta_{\absrefone}) >0$ such that 
the following holds
whenever $n \geq n_{\absrefone}$ and $0 < \rho < \rho_{\absrefone}$. 

If $H$ is an $n$-vertex $(\rho,d_{\absrefone})_{\cherry}$-dense $3$-graph satisfying $\delta(H) \geq \alpha_{\absrefone} \binom{n-1}{2}$, then it admits a $(d_{\absrefone}-\eta_{\absrefone},\mu_{\absrefone},\kappa_{\absrefone})$-absorbing-path. 
\end{lemma}

For the next pillar lemma, $\cherry\;$-denseness is not required. Here, a weaker notion of {\sl denseness} suffices. Let $d,\rho \in (0,1]$ and let $H$ be an $n$-vertex $3$-graph. If 
\begin{equation}\label{eq:set-dense}
e_H(X) := \left| E(H) \cap \binom{X}{3} \right| \geq d\binom{|X|}{3} -\rho n^3
\end{equation}
holds for every $X \subseteq V(H)$, then $H$ is said to be $(\rho,d)$-{\em dense}. If $\rho$ and $d$ are known to exist yet are not made explicit, then we say that $H$ is $1${\em-set-dense}\footnote{If in addition to~\eqref{eq:set-dense} $H$ also satisfies its corresponding upper bound then $H$ is $\points\;$-quasirandom (see e.g.~\cite{weak}).}.  The following lemma imposes no minimum degree conditions on the $3$-graph. It will be used in the proofs of both Theorem~\ref{thm:main} and Theorem~\ref{thm:main2}.

\begin{lemma}\label{lem:path-cover}{\em\bf (Path-cover lemma)}
For every $d_{\pcref},\zeta_{\pcref} \in (0,1]$,  there exist $n_{\pcref} := n_{\pcref}(d_{\pcref},\zeta_{\pcref})$,  $\rho_{\pcref} = \rho_{\pcref}(d_{\pcref},\zeta_{\pcref}) >0$, and an integer $\ell_{\pcref} = \ell_{\pcref}(d_{\pcref},\zeta_{\pcref})$ such that the following holds for all $n \geq n_{\pcref}$ and $0<\rho < \rho_{\pcref}$. 

Let $H$ be an $n$-vertex $(\rho,d_{\pcref})$-dense $3$-graph. Then, all but at most $ \zeta_{\pcref} n$ vertices of $H$ can be covered using at most $\ell_{\pcref}$ vertex-disjoint paths. 
\end{lemma}

The fourth and last pillar lemma is the reservoir 
lemma. We employ the {\sl all encompassing}\footnote{All encompassing in the sense that it can handle the $1$-degree and $2$-degree settings in one stroke.} reservoir lemma of~\cite{Mathias3} which will service both Theorem~\ref{thm:main} and Theorem~\ref{thm:main2}.

\begin{lemma}\label{lem:res-2}{\em\bf(Reservoir lemma)~\cite[Lemma~3.10]{Mathias3}}
Let $U_1,\ldots,U_s$ be subsets of an $n$-element set $V$ and let $L_1,\ldots,L_k$ be graphs on $V$, where $s :=s(n)$ and $k:=k(n)$ are both polynomials in $n$ and such that for sequences of constants $\big(\alpha_i \in (0,1)\big)_{i \in [s]}$ and $\big(\beta_j \in (0,1)\big)_{j \in [k]}$, we have $|U_i| \geq \alpha_i n$ for every $i \in [s]$, and $e(L_j) \geq \beta_j\binom{n}{2}$ for every $j \in [k]$. 

Then, for every constant $\nu_{\resrefone} \in (0,1)$, there exists an $n_{\resrefone} := n_{\resrefone}(\nu_{\resrefone})$ such that if $n \geq n_{\resrefone}$, then there exists a subset $R \subseteq V$ satisfying
\begin{enumerate}
	\item [\namedlabel{itm:R1}{(R.1)}] $\big| |R| - \nu n\big| \leq \nu n^{2/3}$,
	\item [\namedlabel{itm:R2}{(R.2)}] for all $i \in [s]$, $|U_i \cap R| \geq 
	\big(\alpha_i - 2n^{-1/3} \big)|R|$ holds, and
	\item [\namedlabel{itm:R3}{(R.3)}] for all $j \in [k]$, $e(L_j[R]) \geq 
	\big(\beta_j - 3n^{-1/3} \big)\binom{|R|}{2}$ holds. 
\end{enumerate}
\end{lemma}

Lemma~\ref{lem:res-2} is somewhat of an overkill as far as Theorem~\ref{thm:main} is concerned; indeed, the proof of the latter relies rather weakly only on~\ref{itm:R1} and~\ref{itm:R2}. The proof of Theorem~\ref{thm:main2}, though, requires the full force of Lemma~\ref{lem:res-2}, so to speak. In particular, it crucially relies on the $n^{-1/3}$-terms seen in Lemma~\ref{lem:res-2}. A reservoir lemma akin to that seen in~\cite{RRS06} indeed suffices for Theorem~\ref{thm:main}. The latter, however, is subsumed by Lemma~\ref{lem:res-2} and thus omitted. 

\vspace{1ex}
\noindent
\scaps{Constants.} For the most part of this account, we tend to keep track over the raw values of the involved constants. We do, however, appeal on occasion to the notation $\ll$ defined above for constants.

\section{Pairs with positive codegree}\label{sec:positive} 
Let $H$ be a $3$-graph, let $\vec G \subseteq V(H) \times V(H)$, and let 
$(u,v) \in V(H) \times V(H)$. We write 
$$
\deg_H(u,v,\vec G):= |\{z \in V(H): ((z,u) \in \vec G \mor (u,z) \in \vec G)  \mand \{z,u,v\} \in E(H)\}|
$$ 
to denote the number of edges $\{z,u,v\} \in E(H)$ for which at least one of the (ordered) pairs $(z,u)$ or $(u,z)$ is present in  the directed graph $\vec G$. 
Note that $\deg_H(u,v,\vec G) \not= \deg_H(v,u,\vec G)$ is possible. In this definition $\vec G$ is treated as an undirected graph and indeed in the sequel we shall also write $\deg_H(u,v,G)$ when $G$ is an undirected graph. 

We find it more convenient to have the following lemma formulated using undirected graphs.  

\begin{lemma}\label{lem:lowdegpair}
Let $d,\alpha$, and $\rho$ be positive reals and let $H$ be a $(\rho,d)_{\cherry}$-dense $n$-vertex $3$-graph. Let $G$ be a graph on $V(H)$, and let 
$Y \subseteq V(G)$ satisfy
\begin{equation}\label{eq:degG}
\deg_G(y) \geq k\quad \text{for all $y \in Y$},
\end{equation}
where $k := k(n)$ is an integer. 
For an integer $\Delta := \Delta(n)$, set 
$$
\vec B_\Delta := \{(y,z) \in Y \times V(H): \deg_H(y,z,G) < \Delta\}.
$$
Then,
$$
|\vec B_{\Delta}| \leq \frac{\rho n^3}{dk-\Delta}.
$$
\end{lemma}

\begin{proof} Let $G_Y \subseteq G$ be the subgraph of $G$ induced by the edges incident to $Y$. Define\footnote{Here, an edge $yy' \in E(G)$ with $y,y'\in Y$ gives rise to two pairs in $\vec G_Y$.} 
$$
\vec G_Y :=\{(v,y): \{v,y\} \in E(G), v \in V(H), y \in Y\} \subseteq V(H) \times Y.
$$
Then,
$$
d \cdot |\P_2(\vec G_Y,\vec B_\Delta)| - \rho n^3 \leq e_H(\vec G_Y,\vec B_\Delta) < |\vec B_\Delta| \cdot \Delta.
$$
Recalling that 
$$
\P_2(\vec G_Y,\vec B_\Delta) = \{(x,y,z): x,z \in V(H), y \in Y, (x,y) \in \vec G_Y, (y,z) \in \vec B_\Delta\},
$$ 
we may write 
$$
|\P_2(\vec G_Y,\vec B_\Delta)| \geq \sum_{y \in Y} \deg_{G_Y}(y) \big|\{(y,z):(y,z) \in \vec B_\Delta\}\big|
\overset{\eqref{eq:degG}}{\geq} k \sum_{y \in Y} \big|\{(y,z):(y,z) \in \vec B_\Delta\}\big| = k|\vec B_\Delta|. 
$$
Then,  
$$
d k |\vec B_\Delta| - \rho n^3 < |\vec B_\Delta| \cdot \Delta
$$
holds; the claim now follows upon isolating $|\vec B_\Delta|$ in the last inequality. 
\end{proof}

For an $n$-vertex $(\rho,d)_{\cherry}$-dense $3$-graph $H$ and a fixed real $\beta >0$, define
\begin{equation}\label{eq:Bbeta}
B_{\beta}:= B_\beta(H) := \left\{\{u,v\} \in \binom{V(H)}{2}: \deg_H(u,v) < \beta n\right\}
\end{equation}
to consist of all unordered pairs of vertices whose codegree is smaller than $\beta n$. An argument akin to setting $G$ to be the complete graph on $V(H)$, $Y = V(H)$, $k=n-1$, and $\Delta = \beta n$ in Lemma~\ref{lem:lowdegpair}, yields an upper bound on $|B_\beta|$. To see this, consider 
$$
d |\P_2(V(H) \times V(H),\vec B_\beta)| - \rho n^3 \leq e_H(V(H) \times V(H),\vec B_\beta) \leq |\vec B_\beta| \cdot \beta n,
$$
where here $\vec B_\beta := \{(u,v),(v,u): \{u,v\} \in B_\beta\}$. 
Then, 
$$
dn|\vec B_\beta| -\rho n^3 \leq |\vec B_\beta| \cdot \beta n,
$$
so that upon isolating $\vec B_\beta$ we arrive at 
$$
2|B_\beta| = |\vec B_\beta| \leq \frac{\rho n^3}{dn-\beta n} = \frac{\rho}{d-\beta}n^2.
$$
In particular, ignoring the factor of $1/2$, we may write
\begin{equation}\label{eq:Bbeta-size-2}
|B_\beta| \leq \frac{\rho}{d-\beta}n^2;
\end{equation}
the latter making sense whenever $\beta < d$. 

A consequence of~\eqref{eq:Bbeta-size-2}, is that the set of edges 
$$
E_{<\beta} := E_{<\beta}(H) := \bigg\{e \in E(H): \exists \{u,v\} \in \binom{e}{2}\; \text{satisfying}\; \deg_H(u,v) <\beta n\bigg\}
$$ 
satisfies 
$$
|E_{<\beta}| \leq |B_\beta| \cdot n \leq \frac{\rho}{d-\beta}n^3. 
$$
The {\sl spanning} subgraph $H_\beta \subseteq H$ induced by $E(H) \sm E_{<\beta}$ then consists only of edges each pair of which has codegree at least $\beta n$ in $H$. On its own, $H_\beta$ may admit no meaningful minimum degree condition. It does, however, satisfy 
\begin{equation}\label{eq:Hbeta}
e_{H_\beta}(\vec G_1,\vec G_2) \geq e_H(\vec G_1,\vec G_2)- |E_{<\beta}| \geq d|\P_2(\vec G_1,\vec G_2)| - \rho\Big(1+\frac{1}{d-\beta}\Big)n^3,
\end{equation}
for all $\vec G_1,\vec G_2 \subseteq V(H) \times V(H) = V(H_\beta) \times V(H_\beta)$. Consequently, upon a judicious choice of constants, $H_\beta$ {\sl inherits} (in the sense of~\eqref{eq:Hbeta}) a certain level of $\cherry\;$-denseness from $H$. This feature arises in the proof of Theorem~\ref{thm:main2} seen in~\S~\ref{sec:proof-main-2}.

\section{Connecting lemmas}\label{sec:connect}

In this section, we prove Lemmas~\ref{lem:connect} and~\ref{lem:con-1deg}. 
In terms of graphs, our approach, for both these lemmas, can be crudely described as follows. In order to connect two prescribed vertices, a sequence of neighbourhoods, called a {\sl cascade}, is cultivated; one from each vertex. This, until these neighbourhoods {\sl expand} so large as to render a certain quasirandomness assumption non-trivial giving rise to numerous "links" between the two sequences of  neighbourhoods. Two paths are then traced backwards from a "link" to the two prescribed vertices through the two sequences of neighbourhoods; all the while maintaining vertex-disjointness of the paths thus traced.

\subsection{Connecting lemma: $2$-degree setting}

In this section, we prove Lemma~\ref{lem:connect} which is the connecting lemma fitting for Theorem~\ref{thm:main}. At the centre of our proof of Lemma~\ref{lem:connect} is the structure of {\sl cascades}; the next section is dedicated to their definition. 

\subsubsection{Cascades}\label{sec:cascades}

Let $n$ be a sufficiently large integer and let $H$ be an $n$-vertex $3$-graph satisfying $\delta_2(H) \geq \beta n$ for some fixed real $\beta \in (0,1]$ independent of $n$ (and such that $\beta n \leq n-2$, naturally).
Fix $x$ and $y$ to be two vertices in $H$. Below we define the tuple 
$$
\C_\beta(x,y):=\big(x,y,N_1(x,y),N_2(x,y),N_3(x,y),G_1(x,y),G_2(x,y),G_3(x,y)\big)
$$
and refer to it as an $\{x,y\}_\beta$-{\em cascade}; with cascades being a term borrowed from~\cite{RRS06}. All members of the above tuple depend on $\beta$ as well; we omit this from the notation though. In what follows, each of these members is defined. In broad terms, for every $i \in [3]$, $N_i(x,y)$ denotes a set of vertices that essentially corresponds to the {\sl $i$th coneighbourhood of the pair \{x,y\}}. The parameters $(G_i(x,y))_{i\in[3]}$ represent certain graphs {\sl between} these coneighbourhoods which will facilitate the tracking of $5$-paths from $N_3(x,y)$ all the way (back) to $\{x,y\}$. 
\vspace{1em}

Let
$
N_1:=N_1(x,y) := N_H(x,y)
$.
The assumption $\delta_2(H) \geq \beta n$ implies that 
\begin{equation}\label{eq:N1}
|N_1| \geq \beta n. 
\end{equation}
Define $G_1:= G_1(x,y)$ to be the (bipartite) graph whose vertex set is $\{y\} \cup N_1$ and whose edges are given by the set 
$\{yz : z \in N_1\}$. To define $N_2:=N_2(x,y)$ and $G_2 := G_2(x,y)$ we proceed in two step. For the first step, set 
\begin{equation}\label{eq:N'2}
N'_2 := N'_2(x,y):= \bigcup_{z \in N_1} N_H(y,z) = \{w \in V(H): \exists z \in N_1 \; \text{s.t.}\; \{y,z,w\} \in E(H)\}.
\end{equation}
Define $G'_2:=G'_2(x,y)$ to be the graph whose vertex set is $N_1 \cup N'_2$ and whose edges are given by the set 
$$
E(G'_2):=\{zz': z \in N_1, z' \in N'_2\cap  N_H(y,z)\} = \{zz': z \in N_1, z' \in V(H),\; \text{and} \; \{y,z,z'\} \in E(H)  \}.
$$
The assumption that $\delta_2(H) \geq \beta n$ implies that 
$\deg_{G'_2}(z) \geq \beta n$ for every $z \in N_1$. Then, 
\begin{equation}\label{eq:G'2}
e(G'_2) \geq \frac{1}{2}\sum_{z \in N_1} \deg_{G'_2}(z) \geq |N_1|\beta n /2 \overset{\eqref{eq:N1}}{\geq} \beta^2 n^2 / 2.
\end{equation}
For the second step towards the definitions of $N_2:=N_2(x,y)$ and $G_2 := G_2(x,y)$, we discard members of $N'_2$ whose degree in $G'_2$ into $N_1$ is "too low" as follows. Set 
$$
N_2^{(\mathrm{low})}:= N_2^{(\mathrm{low})}(x,y):= \{z \in N'_2: \deg_{G'_2}(z) < \log n\}.
$$
(The choice of $\log n$ here is completely arbitrary. Any function $\omega(n) \ll n$ growing slowly to $\infty$ will suffice; this will become clear soon). 
Setting $N_2 := N_2(x,y) := N'_2 \sm N_2^{(\mathrm{low})}$, we arrive at 
\begin{align*}
\beta^2 n^2 / 2 \overset{\eqref{eq:G'2}}{\leq} e(G'_2) & \leq (\log n) \cdot |N_2^{(\mathrm{low})}| + |N_2|\cdot |N_1|  \leq n \log n + |N_2|\cdot n
\end{align*}
so that for a sufficiently large $n$
\begin{equation}\label{eq:N2}
|N_2| \geq \beta^2 n / 4. 
\end{equation}
Set $G_2:=G_2(x,y) := G'_2[N_1 \cup N_2]$. This concludes the definitions of $N_2:=N_2(x,y)$ and $G_2 := G_2(x,y)$. 


We turn to the definition of the set $N_3 := N_3(x,y)$ and the graph $G_3 := G_3(x,y)$. To that end, associate an auxiliary graph $\B_w:= \B_w(x,y)$ with every vertex $w \in N_2$. In particular, for a fixed vertex $w \in N_2$, let $\B_w$ be the graph whose vertex set is $V(H)$ and whose edges are given by the set
$$
E(\B_w) := \{uz: u \in V(H), z \in N_{G_2}(w) \subseteq N_1,\; \text{and}\; \{z,w,u\} \in E(H)\}.
$$ 
Define 
$$
N_3:= N_3(x,y) := \{u \in V(H): \exists w \in N_2 \; \text{s.t.} \; \deg_{\B_w}(u) \geq 20\}
$$
and let $G_3:= G_3(x,y)$ be the graph whose vertex set is $N_2 \cup N_3$ and whose edge set is given by 
$$
E(G_3) := \{u w: u \in N_3, w \in N_2, \; \text{and} \; \deg_{\B_w}(u) \geq 20\}.
$$
This completes the definition of an $\{x,y\}_\beta$-cascade. 
\vspace{1em}

We conclude this section by recording a few useful traits of $\{x,y\}_\beta$-cascades that will be called upon in subsequent arguments. Continuing with the notation set thus far, fix $w \in N_2$. Then, 
$$
\deg_{G_2}(w) \cdot \beta n \overset{\delta_2(H) \geq \beta n}{\leq} e(\B_w) \leq 20 \cdot s + (n-s) \deg_{G_2}(w),
$$
where $s$ denotes the number of vertices $u \in V(H)$ satisfying $\deg_{\B_w}(u) < 20$. Then, for a sufficiently large  $n$
$$
n-s \geq \beta n - \frac{20 \cdot s}{\deg_{G_2}(w)} \geq \beta n - \frac{20 \cdot n}{\log n} \geq \beta n /2
$$
holds. As $\deg_{G_3}(w) = n-s$ it follows that 
\begin{equation}\label{eq:G3N2deg}
\deg_{G_3}(w) \geq \beta n / 2 \; \text{for every $w \in N_2$};
\end{equation}
this, in particular, implies that
\begin{equation}\label{eq:G3N3}
|N_3| \geq \beta n / 2 \; \text{and} \; e(G_3) \geq \frac{1}{2} \sum_{w \in N_2} \deg_{G_3}(w) \geq |N_2| \beta n / 4 \overset{\eqref{eq:N2}}{\geq} \beta^2 n ^2 /8
\end{equation}

\subsubsection{Links}

In addition to $\{x,y\}$ and $\C_\beta(x,y)$ defined in~\S~\ref{sec:cascades}, let $\{x',y'\}$ be a pair of vertices disjoint from $\{x,y\}$, and let $\C_\beta(x',y')$ be an $\{x',y'\}_\beta$-cascade in $H$. A quadruple  $(z,u,v,w) \in N_2(x,y) \times N_3(x,y) \times N_3(x',y') \times N_2(x',y')$ is said to be an $(\{x,y\},\{x',y'\})$-{\em link} with respect to $\C_\beta(x,y)$ and $\C_\beta(x',y')$, if  
\begin{enumerate}
	\item [(L.1)] $x,y,z,u,v,w,y',x'$ are all distinct. 
	\item [(L.2)] $\{z,u,v\}, \{u,v,w\} \in E(H)$, and 
	\item [(L.3)] $zu \in E(G_3(x,y))$ and $vw \in E(G_3(x',y'))$. 
\end{enumerate}

\begin{lemma}\label{lem:link}
If two distinct pairs of vertices namely $\{x,y\}$ and $\{x',y'\}$ admit an $(\{x,y\},\{x',y'\})$-link, then $H$ admits a $10$-path connecting $\{x,y\}$ and $\{x',y'\}$. 
\end{lemma}

\begin{proof}
Let $(z,u,v,w) \in N_2(x,y) \times N_3(x,y) \times N_3(x',y') \times N_2(x',y')$ be an $(\{x,y\},\{x',y'\})$-link. First we construct a $5$-path connecting $\{x,y\}$ and $\{z,u\}$ through $\C(x,y)$. Having $zu \in E(G_3(x,y))$ means that $\deg_{\B_z}(u) \geq 20$, which in other words means that there are at least $20$ vertices $z' \in N_{G_2(x,y)}(z) \subseteq N_1(x,y)$ such that $\{z',z,u\} \in E(H)$. We may then choose one such vertex $z'$ such that $z' \in N_1(x,y) \sm \{x,y,x',y',z,u,v,w\}$. Having $zz' \in E(G_2(x,y))$ implies that $\{y,z',z\} \in E(H)$. The $5$-path is made complete with the fact that $\{x,y,z'\} \in E(H)$. Let $P$ denote this path. 

It remains to construct a $5$-path through the cascade of $\{x',y'\}$ connecting $\{v,w\}$ and $\{x',y'\}$ in such a way as to not meet any vertex of $P$. The same argument used for constructing $P$ can be used here as well except for one change. This time around, we require a vertex $z'' \in N_2(x',y')$ to play the corresponding role assumed by $z'$ above. The vertex $z''$ must satisfy $z'' \notin \{x,y,z',x',y',z,u,v,w\}$ (i.e., it has to avoid $z'$ as well). Clearly there is enough freedom to do so.
\end{proof}

\subsubsection{Proof of Lemma~\ref{lem:connect}}

Given $d := d_{\conref}, \alpha := \alpha_{\conref}$ as in the premise of the lemma, set
\begin{equation}\label{eq:rho1}
\rho_{\conref}(d,\alpha) :=  \frac{d \alpha^6}{2^{13}}.
\end{equation}
Let $0<\rho < \rho_{\conref}(d,\beta)$ be fixed, let $H$ be a $(\rho,d)_{\cherry}$-dense $3$-graph satisfying $\delta_2(H) \geq \alpha n$ (we naturally assume that $\alpha$ is such that $\alpha n \leq n-2$), and let $\{x,y\}$ and $\{x',y'\}$ be two disjoint pairs of vertices in $V(H)$.

By Lemma~\ref{lem:link} it suffices to show that the cascades $\C_\alpha(x,y)$ and $\C_\alpha(x',y')$ taken in $H$ admit an $(\{x,y\},\{x',y'\})$-link (in $H$). Owing to~\eqref{eq:G3N3}, $e(G_3(x,y)) \geq \alpha^2n^2/8$. There exists a subgraph $F \subseteq G_3(x,y)$ satisfying $\delta(F) \geq \alpha^2 n / 8$ (see, e.g., \cite[Proposition~1.2.2]{Diestel}). Then, 
\begin{equation}\label{eq:FN3}
|V(F) \cap N_3(x,y)| \geq \alpha^2 n / 8.
\end{equation} 
Indeed, all edges in $G_3(x,y)$ (and thus in $F$) are of the form $N_2(x,y) \times N_3(x,y)$ (though $N_2(x,y) \cap N_3(x,y)$ need not be empty). Hence, there is a vertex $w \in N_2(x,y)  \cap V(F)$. By definition $N_F(w) \subseteq N_{G_3(x,y)}(w) \subseteq N_3(x,y)$ and~\eqref{eq:FN3} follows.

Set
$$
\vec B: = \big\{(z,u) \in \big(V(F) \cap N_3(x,y)\big) \times V(H)  : z \not= u\; \text{and}\; \deg_H(z,u,F)  < d \alpha^2 n / 16\big\}.
$$
By Lemma~\ref{lem:lowdegpair} applied with 
$G = F$, $
Y:= V(F) \cap N_3(x,y)$,  $k = \alpha^2 n /8$, and $\Delta := d \alpha^2 n /16
$
we attain
$$
|\vec B|  \leq \frac{\rho n^3}{d \alpha^2 n / 8 - d \alpha^2 n /16} 
 = \frac{16 \rho }{d \alpha^2} \cdot n^2.
$$
A symmetrical argument applied to $\C_\alpha(x',y')$ asserts that the set 
$$
\vec B' := \big\{(z,u) \in \big(V(F') \cap N_3(x',y')\big) \times V(H) : z \not= u \; \text{and}\; \deg_H(z,u,F')  < d \alpha^2 n / 16\big\}
$$
satisfies 
$
|\vec B'| \leq \frac{16 \rho }{d \alpha^2} \cdot n^2
$ 
as well,
where here $F' \subseteq G_3(x',y')$ is the counterpart of $F$ in this argument (i.e. it is a subgraph of $G_3(x',y')$ satisfying $\delta(F') \geq \alpha^2 n / 8$).

The set $(V(F) \cap N_3(x,y)) \times (V(F') \cap N_3(x',y'))$ has size at least $\alpha^4 n^2/2^6$, by~\eqref{eq:FN3}; removing degenerate members (i.e., members of the form $(x,x)$), we retain at least $\alpha^4 n^2 /2^7$ non-degenerate members of that Cartesian product. The latter set of non-degenerate pairs, we denote by $\vec T$. 
Then,
$$
|\vec T \sm ( \vec B \cup \vec B')| \geq |\vec T| - |\vec B| - |\vec B'|\geq \frac{\alpha^4}{2^7}n^2 - \frac{32 \rho}{d \alpha^2}n^2 \overset{\eqref{eq:rho1}}{\geq} \frac{\alpha^4}{2^8}n^2
$$
holds. Each member $(u,v) \in \vec T \sm (\vec B \cup \vec B')$ satisfies $u \not= v$, $u \in V(F) \cap N_3(x,y)$, $v \in V(F') \cap N_3(x',y')$, $\deg_G(u,v,F) \geq d\alpha^2 n /16$, and $\deg_G(v,u,F') \geq d\alpha^2 n /16$. That is, there are at least $d\alpha^2 n /16$ edges $\{u,v,z\} \in E(H)$ with $uz \in E(F)$ (so that $z \in N_2(x,y)$) and at least $d\alpha^2 n /16$ edges $\{u,v,w\} \in E(H)$ with $vw \in E(F')$ (so that $w \in N_2(x',y')$). Hence, for a sufficiently large $n$, we may insist on (many choices) $w \not= z$ and thus form the required $\{(x,y),(x',y')\}$-link.

This completes the proof of Lemma~\ref{lem:connect}.

\subsection{Connecting lemma: $1$-degree setting}

In this section, we prove Lemma~\ref{lem:con-1deg} which is the connecting lemma fitting for Theorem~\ref{thm:main2}. 
The definition of cascades, seen at~\S~\ref{sec:cascades}, fits any $3$-graph $H$ satisfying $\delta_2(H) = \Omega(n)$. As such the construction of cascades makes no appeal to $\cherry\;$-denseness. In Lemma~\ref{lem:con-1deg}, which is furnished with a minimum $1$-degree condition only, cascades, as defined, are not at our disposal (at least not verbatim). To prove Lemma~\ref{lem:con-1deg} then, we put forth a definition of a structure to which we refer as {\sl refined cascades}. The latter is an adaption of cascades to the setting of Lemma~\ref{lem:con-1deg}. 

While we do follow closely the definition of cascades when defining their refined counterparts, these two structures are quite different from one another. One crucial manifestation of this difference can be seen through the condition $\alpha +d > 1$ stated in the premise of Theorem~\ref{thm:main2}. This condition is, in fact, incurred through the definition of {\sl refined cascades}. The construction of latter is then the sole "bottleneck" in our approach preventing us from establishing Conjecture~\ref{conj:main}. 

Unlike the case of cascades, the construction of their refined counterparts does make appeals to $\cherry\;$-denseness of the host $3$-graph. Consequently, the definitions of cascades and refined cascades are not consolidated.



\subsubsection{Refined cascades}
Let $\alpha,d$, and $\eta$ satisfying $\alpha +d > 1$ and 
\begin{equation}\label{eq:eta-cond}
0<\eta < g(\alpha,d) = \min\{\alpha,\;d,\;(\alpha +d -1)/(1+d)\} 
\end{equation}
be given. The function $g(\cdot)$ is the one defined in~\eqref{eq:g-func}. 
An appeal to the inequality $\alpha + d > 1$ is made in the numerator of the third term appearing on the r.h.s. of~\eqref{eq:eta-cond}. Set an auxiliary constant 
\begin{equation}\label{eq:zeta-cas}
0 < \zeta := \zeta(\alpha,d,\eta) < \eta +\eta d - \eta^2.
\end{equation}
Set
\begin{equation}\label{eq:rho-cas}
0< \rho < \min\{\eta/4,\;\eta\zeta(d-\eta)/4\}.
\end{equation}

Let $H$ be a $(\rho,d)_{\cherry}$-dense $3$-graph satisfying $\delta(H) \geq \alpha \binom{n-1}{2}$. Setting $\beta = d- \eta$ in~\eqref{eq:Bbeta-size-2}, yields
\begin{equation}\label{eq:Bdeta}
|B_{d-\eta}| \leq \frac{\rho}{d-(d-\eta)}n^2 = \frac{\rho}{\eta}n^2. 
\end{equation}
Let $\{x,y\} \in \binom{V(H)}{2}$ satisfying $\deg_H(x,y) \geq (d-\eta)n$ be fixed; 
by~\eqref{eq:Bdeta} coupled with the condition $\rho < \eta/4$ stipulated in~\eqref{eq:eta-cond}, there are $\Omega(n^2)$ such pairs in $H$.

In what follows, we define the tuple 
$$
\R_{\alpha,d,\eta}(x,y)  := (x,y,\N_1(x,y),\N_2(x,y),\N_3(x,y), \G_1(x,y),\G_2(x,y),\G_3(x,y))
$$
and refer to it as an $\{x,y\}_{\alpha,d,\eta}$-{\em refined-cascade}. All members of this tuple depend on $\alpha,d$, and $\eta$ as well, yet we omit this from the notation. The members $\N_i$ and $\G_i$, $i \in [3]$, assume roles analogous to those assumed by $N_i$ and $G_i$, $i \in [3]$, in the definition of cascades in~\S~\ref{sec:cascades}.
We proceed with the definition of each of the members of the above tuple. \vspace{1em}

Define 
\begin{equation}\label{eq:intersection}
\N_1:= \N_1(x,y) := N_H(x,y) \cap \{z \in V(H) : \deg_H(z,y) \geq \eta n\} 
\end{equation}
Treating $B_\eta$ (see~\eqref{eq:Bbeta} for a definition) as a graph, we write 
$$
\overline{B}_\eta := \Big\{\{u,v\} \in \binom{V(H)}{2}: \deg_H(u,v) \geq \eta n \Big\}
$$
to denote the graph complementing $B_\eta$ over $V(H)$. The set $\{z \in V(H) : \deg_H(z,y) \geq \eta n\}$, appearing in~\eqref{eq:intersection}, is then the neighbourhood of $y$ in $\overline{B}_\eta$. One is now reminded of the following remarkable fact established in~\cite[Claim~3.1]{Mathias3}, which in our setting (and owing to $\eta < \alpha$ as imposed in~\eqref{eq:eta-cond}) reads as follows.

\begin{claim}\label{clm:mathias}{\em~\cite[Claim~3.1]{Mathias3}}
$\delta(\overline{B}_\eta) \geq \frac{\alpha-\eta}{1-\eta}(n-1)$.
\end{claim}
 
\noindent
By definition of $\{x,y\}$, $|N_H(x,y)| \geq (d-\eta)n$; this together with Claim~\ref{clm:mathias} collectively imply that if 
$$
\frac{\alpha -\eta}{1-\eta} + d-\eta > 1, 
$$
then $|\N_1|=\Omega(n)$. 
Rewriting this inequality as 
\begin{equation}\label{eq:a-d-1}
\alpha +d - 1 > \eta +\eta d - \eta^2,
\end{equation}
we note that the inequality $\eta < \frac{a+d-1}{1+d}$ imposed in~\eqref{eq:eta-cond} implies that~\eqref{eq:a-d-1} is satisfied. Moreover, it is here at~\eqref{eq:a-d-1} that the condition $\alpha +d > 1$, imposed in Theorem~\ref{thm:main2}, stands out. 
It follows that 
\begin{equation}\label{eq:N1-size-2}
|\N_1| \overset{\eqref{eq:zeta-cas}}{\geq} \zeta n.
\end{equation}
Define $\G_1 := \G_1(x,y)$ to be the (bipartite) graph whose vertex set is given by $\{y\} \cup \N_1$ and whose edge set is given by $\{yz : z \in \N_1\}$. \vspace{1em}

We proceed to defining $\N_2(x,y)$ and $\G_2(x,y)$. Set 
$$
\N'_2 := \N'_2(x,y) := \bigcup_{z \in \N_1} N_H(y,z) = \{w \in V(H): \exists z \in \N_1\; \text{s.t.}\; \{y,z,w\} \in E(H)\},
$$
and define $\G'_2:= \G'_2(x,y)$ to be the graph whose vertex set is $\N_1 \cup \N_2$ and whose edge set is given by 
$$
E(\G'_2) := \{zz': z \in \N_1, z' \in V(H), \{y,z,z'\} \in E(H)\}.
$$
By definition of $\N_1$ (see~\eqref{eq:intersection}), $\deg_H(y,z) \geq \eta n$ for every $z \in \N_1$ so that $\deg_{\G'_2}(z) \geq \eta n$ holds for every $z\in \N_1$. Then, 
\begin{equation}\label{eq:G'2-size-2}
e(\G'_2) \geq \frac{1}{2} \sum_{z \in \N_1} \deg_{\G'_2}(z) \geq |\N_1| \eta n /2 \overset{\eqref{eq:N1-size-2}}{\geq} \zeta \eta n^2 /2. 
\end{equation}
All but at most $\frac{\rho}{d-\eta}n^2$ of the edges of $\G'_2$ lie in $B_\eta$, by~\eqref{eq:Bbeta-size-2} (with $\beta = \eta$ in that equation). Consequently, there exists a subgraph $\G''_2 \subseteq \G'_2$ satisfying 
$$
e(\G''_2) \overset{\eqref{eq:G'2-size-2}}{\geq} \big(\frac{\eta \zeta}{2} - \frac{\rho}{d-\eta}\big) n^2 \overset{\eqref{eq:rho-cas}}{\geq} \frac{\eta\zeta}{4}n^2,
$$
having the property that $E(\G''_2) \cap B_\eta =\emptyset$. Then, by~\cite[Proposition~1.2.2]{Diestel}, there exists a subgraph $\G_2 \subseteq \G''_2$ satisfying $\delta(\G_2) \geq \frac{\eta\zeta}{4}n$ and this completes the definition of $\G_2$. We conclude this part of the definition by setting $\N_2 := V(\G_2) \cap \N'_2$. The property $\delta(\G_2) \geq \frac{\eta\zeta}{4}n$, together with the fact that all edges of $\G_2$ are of the form $\N_1 \times \N_2$ imply that 
\begin{equation}\label{eq:N2-size-2}
|\N_2| \geq \zeta \eta n/4.  
\end{equation}

Next, we define $\N_3(x,y)$ and $\G_3(x,y)$.
For $w \in \N_2$, let $\X_w$ be the graph on $V(H)$ whose edge set is given by 
$$
E(\X_w) := \{uz: u \in V(H), z \in N_{\G_2}(w) \subseteq \N_1, \; \text{and}\; \{z,w,u\}\in E(H)\}. 
$$ 
Define 
$$
\N_3:= \N_3(x,y) := \{u \in V(H): \exists w \in \N_2\; \text{s.t.} \; \deg_{\X_w}(u)\geq 20\},
$$
and let $\G_3:=\G_3(x,y)$ be the graph whose vertex set is $\N_2 \cup \N_3$ and whose edge set is given by 
$$
E(\G_3) := \{uw: u \in \N_3, w \in \N_2, \; \text{and}\; \deg_{\X_w}(u) \geq 20\}.
$$
Then,
$$
\deg_{\G_2}(w) \eta n \leq e(\X_w) \leq 20 \cdot r + (n-r) \deg_{\G_2}(w) = 20 \cdot r + \deg_{\G_3}(w)\deg_{\G_2}(w)
$$
holds, where here $r$ denotes the number of vertices $u \in V(H)$ satisfying $\deg_{\X_w}(u) < 20$; the first inequality is owing to $E(\G_2) \cap B_\eta =\emptyset$, by definition of $\G_2$, and the last equality is owing to $\deg_{\G_3}(w)=n-r$, by definition of $r$. We may then write that 
$$
\deg_{\G_3}(w) \geq \eta n - \frac{20\cdot r}{\deg_{\G_2}(w)} \overset{\delta(\G_2) \geq \zeta\eta n/4}{\geq} \eta n - \frac{20\cdot r}{\frac{\zeta \eta}{4} n} \overset{r \leq n}{\geq} \frac{\eta}{2} n,
$$ 
where the last inequality is assuming $n$ is sufficiently large. Consequently,
\begin{equation}\label{eq:G3-size-2}
|\N_3| \geq \frac{\eta}{2} n \mand e(\G_3) \geq \frac{1}{2} \sum_{w \in \N_2} \deg_{\G_3}(w) \geq |\N_2| \frac{\eta}{4}n \overset{\eqref{eq:N2-size-2}}{\geq} \frac{\zeta \eta^2}{16}{n^2}. 
\end{equation}

This concludes the definition of refined cascades and properties thereof.

\subsubsection{Proof of Lemma~\ref{lem:con-1deg}}

With the definition of refined cascades complete, our proof of Lemma~\ref{lem:con-1deg} follows closely that seen for Lemma~\ref{lem:connect}.  Indeed, the machinery of {\sl links} defined for cascades does carry over to refined cascades essentially verbatim. 

\begin{proofof}{Lemma~\ref{lem:con-1deg}}
Given $d:=d_{\conrefone}$, $\alpha := \alpha_{\conrefone}$, and $\eta:= \eta_{\conrefone}$ as in the premise of Lemma~\ref{lem:con-1deg}, set an auxiliary constant $\zeta$ satisfying~\eqref{eq:zeta-cas}, and put 
\begin{equation}\label{eq:rho-con-2}
\rho_{\conrefone}(d,\alpha,\eta) := \min\{\eta/4,\;\eta\zeta(d-\eta)/4,\;d\zeta^4\eta^6/2^{16}\}. 
\end{equation}
Let $0<\rho < \rho_{\conrefone}(d,\alpha,\eta)$ be fixed, and let $H$ be a $(\rho,d)_{\cherry}$-dense $3$-graph satisfying $\delta(H) \geq \alpha \binom{n-1}{2}$. Let $\{x,y\}$ and $\{x',y'\}$ be two disjoint pairs of vertices in $V(H)$, each having codegree at least $(d-\eta)n$ (existence of such pairs is established in~\eqref{eq:Bdeta} and explanation thereafter). 

By Lemma~\ref{lem:link} it suffices to show that the refined cascades $\R_{\alpha,d,\eta}(x,y)$ and $\R_{\alpha,d,\eta}(x',y')$ admit an $(\{x,y\},\{x',y'\})$-link. 
By~\eqref{eq:G3-size-2} and~\cite[Proposition~1.2.2]{Diestel}, there exists a subgraph $F \subseteq \G_3$ satisfying $\delta(F) \geq \frac{\zeta\eta^2}{16}n$. Then, 
\begin{equation}\label{eq:F-2}
|V(F) \cap \N_3(x,y)| \geq \frac{\zeta\eta^2}{16}n. 
\end{equation}
Set 
$$
\vec B:= \big\{(z,u) \in \big(V(F) \cap \N_3(x,y)\big) \times V(H): z \not= u \mand \deg_H(z,u,F) < d\zeta \eta^2 n /32 \big\},
$$
and note that 
$$
|\vec B| \leq \frac{\rho n^3}{d\zeta\eta^2n/16 - d \zeta \eta^2/32} = \frac{32 \rho}{d\zeta \eta^2}n^2,  
$$
holds, by Lemma~\ref{lem:lowdegpair}.
A symmetrical argument applied to $\R_{\alpha,\eta}(x',y')$ asserts that the set 
$$
\vec B' := \big\{(z,u) \in \big(V(F') \cap N_3(x',y')\big) \times V(H) : z \not= u \; \text{and}\; \deg_H(z,u,F')  < d\zeta \eta^2 n /32\big\}
$$
satisfies 
$
|\vec B'| \leq \frac{32 \rho}{d\zeta \eta^2}n^2
$ as well, where $F' \subseteq \G_3(x',y')$ is the counterpart of $F$.  

The set $\big(V(F) \cap \N_3(x,y)\big) \times \big(V(F') \cap \N_3(x',y')\big)$ has size at least $\zeta^2 \eta^4 n^2/2^8$, by~\eqref{eq:F-2}; removing degenerate members (i.e., members of the form $(x,x)$) we retain at least $\zeta^2 \eta^4 n^2 /2^9$ non-degenerate members of that Cartesian product. The latter set of non-degenerate pairs, we denote by $\vec T$. 
Then, for a sufficiently large $n$
$$
|\vec T \sm ( \vec B \cup \vec B')| \geq \frac{\zeta^2\eta^4}{2^9}n^2 - \frac{64 \rho}{d\zeta \eta^2}n^2 \overset{\eqref{eq:rho-con-2}}{\geq} \frac{\zeta^2 \eta^4}{2^{10}}n^2.
$$
holds, and the lemma follows.   
\end{proofof}

\subsection{A connecting lemma with no minimum degree conditions}\label{sec:useless}

In~\S~\ref{sec:approach} we mentioned that for $\cherry\;$-dense graphs, a connecting lemma imposing no minimum degree conditions is possible and that such a lemma is insufficient for our needs. In this section, we make this precise. 

Given an $n$-vertex $3$-graph $H$ and a real $\beta >0$, define the sequence $H=: H_0 \supseteq H_1 \supseteq H_2 \cdots$ of spanning subgraphs of $H$ as follows. At step $i$, if every pair $\{u,v\}$ of vertices of $H_i$ satisfies either $\deg_{H_i}(u,v) \geq \beta n$ or $\deg_{H_i}(u,v) = 0$, then stop. Otherwise, $H_i$ admits a pair of vertices $\{u,v\}$ satisfying $0 < \deg_{H_i}(u,v) < \beta n$. In which case, remove all edges of $H_i$ containing the pair $\{u,v\}$ and denote the resulting (spanning) subgraph of $H$ by $H_{i+1}$. 

As overall there are $\binom{n}{2}$ pairs of vertices to consider, then throughout the above process a total of at most $\beta n^3$ of the edges of $H$ are removed. Consequently, if $\beta$ and $H$ are such that $e(H) > \beta n^3$, then the above process terminates in a non-empty spanning subgraph of $H$, denoted $H^{(\beta)}$. The latter has the property that either $\deg_{H^{(\beta)}}(u,v) \geq \beta n$ or $\deg_{H^{(\beta)}}(u,v) = 0$, whenever $\{u,v\}$ is a pair of vertices of $H$. Put another way, any pair of vertices $\{u,v\}$ captured by an edge of $H^{(\beta)}$ satisfies $\deg_{H^{(\beta)}}(u,v) \geq \beta n$. More concisely, one may write
$$
\delta^*_2\big(H^{(\beta)}\big) := \min \left\{\deg_{H^{(\beta)}}(u,v): \{u,v\} \in \binom{V\big(H^{(\beta)}\big)}{2} \mand \exists e \in E\big(H^{(\beta)}\big)\; s.t.\; \{u,v\} \subseteq e \right\} \geq \beta n.
$$
Pairs of vertices captured by the edges of $H^{(\beta)}$ are termed $\beta$-{\em relevant}. Such pairs of vertices, taken in $H^{(\beta)}$, can be connected in $H$ using the following lemma.    

\begin{lemma}\label{lem:connect-useless}
For every $d, \beta \in (0,1]$ such that $\beta < d$, there exist an integer $n_0>0$ and a real $\rho_0  >0$ such that the following holds for all $n \geq n_0$ and 
$0 < \rho < \rho_0$. 

Let $H$ be an $n$-vertex $(\rho,d)_{\cherry}$-dense $3$-graph and let $\{x,y\}$ and $\{x',y'\}$ be two disjoint $\beta$-relevant pairs of vertices. Then, there exists a $10$-path in $H$ connecting $\{x,y\}$ and $\{x',y'\}$. 
\end{lemma}

The proof of Lemma~\ref{lem:connect-useless} is that of Lemma~\ref{lem:connect} essentially verbatim. Let the two $\beta$-relevant pairs $\{x,y\}$ and $\{x',y'\}$ per Lemma~\ref{lem:connect-useless} be given. First, construct the cascades $\C_\beta(x,y)$ and $\C_\beta(x',y')$ in $H^{(\beta)}$ (instead of $H$) while  throughout the construction of these replace every appeal to $\delta_2\big(H^{(\beta)}\big)$ (which may be zero) with an appeal to $\delta^*_2\big(H^{(\beta)}\big)$. Indeed, the construction of cascades only requires a sufficiently large minimum codegree for the pairs already captured through the edges of the cascades and in this manner one progresses from one level of the cascade to the next. Second, with these cascades constructed note that these exist in $H$ and thus an $(\{x,y\},\{x',y'\})$-link can be found in $H$ using the very same argument seen in the proof of Lemma~\ref{lem:connect} for that stage. 

Unfortunately, we were unable to employ Lemma~\ref{lem:connect-useless} in our account. Indeed, in subsequent arguments the connecting lemmas are used repeatedly 
in order to connect prescribed pairs of vertices which although admit a relatively large codegree are essentially arbitrary. We were unable to determine whether these pairs are also $\beta$-relevant (for an appropriate $\beta$). For indeed, a pair is $\beta$-relevant if it manages to survive the {\sl cleanup procedure}, so to speak, giving rise to $H^{(\beta)}$. Arbitrary pairs of vertices admitting high codegree in $H$ may of course not survive this process. 

We do, however, perceive Lemma~\ref{lem:connect-useless} as being relevant to the pursuit of Conjecture~\ref{conj:main} and consequently mention it here.

\section{Absorbing-path lemmas}\label{sec:absorb}

In this section, we prove Lemmas~\ref{lem:path-absorb} and~\ref{lem:path-absorb-2}. At the core of these proofs stands the notion of a $\beta$-{\sl absorber} which is a variant of what is often referred to as the {\em natural absorber} as far as tight cycles in $3$-graphs are concerned. 

\begin{definition}
Let $H$ be a $3$-graph. For $\beta >0$ and $v \in V(H)$, a quadruple $(x,y,z,w) \in V(H)^4$ is said to be a $(\beta,v)$-{\em absorber} if 
\begin{enumerate}
\item [(A.1)] $
\{x,y,z\},\{y,z,w\}, \{v,x,y\},\{v,y,z\},\{v,z,w\} \in E(H). 
$
\item [(A.2)] Moreover, $\deg_H(x,y),\deg_H(z,w) \geq \beta n$. 
\end{enumerate}
\end{definition}

\noindent
We say $\beta$-{\em absorber} to mean $(\beta,v)$-absorber for some $v \in V(H)$.

Our proofs of both absorbing-path lemmas are modelled after the same conceptual three step argument seen in~\cite{RRS06}. First, a counting lemma for $(\beta,v)$-absorbers with the vertex $v$ prescribed is established; this can be seen in Lemma~\ref{lem:cnt}. Second, the aforementioned counting lemma is employed in a hypergeometric experiment as to establish the existence of a "small" set $\F$ of vertex-disjoint $\beta$-absorbers that can absorb 
any set of vertices that is not too "large"; this can be seen in Lemma~\ref{lem:F}. Third, using the connecting lemmas, namely Lemmas~\ref{lem:connect} and~\ref{lem:con-1deg}, we "string", so to speak, the members of $\F$ into a single path yielding the required absorbing path. 

Lemmas~\ref{lem:cnt} and~\ref{lem:F}, capturing the first two steps in the above outlined approach, are capable of handling both settings considered in Theorems~\ref{thm:main} and~\ref{thm:main2}. This is due to~\cite[Remark~1.4]{RR10} asserting that if an $n$-vertex $3$-graph $H$ admits $\delta_2(H) = \Omega(n)$, then $\delta(H) = \Omega(n^2)$.

The third step in the above plan, however, we treat separately across the two aforementioned settings. While the overall scheme of the third step is the same between the two settings, it is here that invocations to the two connecting lemmas are made. The inherent differences between these two lemmas compels (us into having) two separate treatments. In~\S~\ref{sec:proof-absorb-1}, Lemma~\ref{lem:path-absorb} is proved, while in~\S~\ref{sec:proof-absorb-2}, Lemma~\ref{lem:path-absorb-2} is proved.  

\subsection{A counting lemma for $\beta$-absorbers}\label{sec:cnt}

Let $H$ be a $3$-graph and let $v \in V(H)$. We write $L_v:=L_v(H)$ to denote the {\em link graph of $v$}, that is the graph whose vertex set is $V(H)\sm \{v\}$ and in which two (distinct) vertices, namely $x$ and $y$, form an edge whenever $\{x,y,v\} \in E(H)$. Put 
$$
L_{\beta,v} := \{xy \in L_v: \deg_H(x,y) \geq \beta n\}.
$$

\begin{lemma}\label{lem:cnt}
For every $d_{\cntref},\alpha_{\cntref}, \beta_{\cntref} \in (0,1]$ such that $\beta_{\cntref} <  d_{\cntref}$, there exist an integer $n_{\cntref}:=n_{\cntref}(d_{\cntref},\alpha_{\cntref}, \beta_{\cntref})$, a real $\rho_{\cntref} = \rho_{\cntref}(d_{\cntref},\alpha_{\cntref},\beta_{\cntref}) > 0$, and a real $c_{\cntref} := c_{\cntref}(d_{\cntref},\alpha_{\cntref})> 0$ such that the following holds for any integer $n \geq n_{\cntref}$ and $0<\rho < \rho_{\cntref}$. 

Let $H$ be an $n$-vertex $(\rho,d_{\cntref})_{\cherry}$-dense $3$-graph satisfying $\delta(H) \geq \alpha_{\cntref} \binom{n-1}{2}$, and let $v \in V(H)$. Then, there are at least 
$
c_{\cntref}n^4
$
$(\beta,v)$-absorbers in $H$.
\end{lemma}

\begin{proof}
Given $\alpha := \alpha_{\cntref}, \beta:=\beta_{\cntref}$ and $d := d_{\cntref}$ set 
\begin{equation}\label{eq:rho-cnt}
\rho_{\cntref} := \min \{\alpha (d-\beta)/8,\;d\alpha^{10}/2^{36}\}.
\end{equation}
let $0<\rho < \rho_{\cntref}$ be fixed, and let $n$ be sufficiently large. Let $H$ be an $n$-vertex $(\rho,d)_{\cherry}$-dense $3$-graph as in the premise and fix $v \in V(H)$. 

Having $\deg_H(v) \geq \alpha \binom{n-1}{2}$ asserts that $e(L_v) \geq \alpha \binom{n-1}{2}$. Then, for a sufficiently large $n$
$$
e(L_{\beta,v}) \geq e(L_v) - |B_\beta| \overset{\eqref{eq:Bbeta-size-2}}{\geq} \alpha \binom{n-1}{2} - \frac{\rho}{d-\beta}n^2\geq \frac{\alpha}{4}n^2 - \frac{\rho}{d-\beta}n^2 \overset{\eqref{eq:rho-cnt}}{\geq} \alpha n^2 /8,
$$
where $B_\beta$ is as in~\eqref{eq:Bbeta}. By~\cite[Proposition~1.2.2]{Diestel}, $L_{\beta,v}$ admits a subgraph with minimum degree at least $\alpha n /8$, implying in turn that 
\begin{equation}\label{eq:Lbv-size}
\ell := |V(L_{\beta,v})| \geq \alpha n /8. 
\end{equation}

Sidorenko's conjecture~\cite{Sid1,Sid2} is true for the $2$-graph $P_4$~\cite{BR}, where by $P_4$ we mean the path consisting of $3$ edges and $4$ vertices. Then, for  sufficiently large $n$ there are at least 
$$
(n-1)^4\left(\frac{2e(L_{\beta,v})}{n^2}\right)^3 \geq \frac{n^4}{2}\cdot  \Big(\frac{2\alpha}{8}\Big)^3 = \frac{\alpha^3}{2^7} n^4
$$
homomorphisms of $P_4$ into $L_{\beta,v}$. Consequently (and again assuming $n$ is sufficiently large) there is a collection $\P$ of at least $\alpha^3 n^4/2^8$ labelled copies of $P_4$ in $L_{\beta,v}$. 

For an ordered pair $(u,w) \in V(L_{\beta,v}) \times V(L_{\beta,v})$, let $P_4(u,w)$ denote the number of members of $\P$ with the form $(x,u,w,y)$. 
Set $K:= \alpha^3 \ell^2 /2^{10}$. Owing to~\eqref{eq:Lbv-size}, 
\begin{equation}\label{eq:Ksize}
K \geq \frac{\alpha^5}{2^{16}}n^2.
\end{equation}
Put
\begin{align*}
\vec X & := \{(u,w) \in V(L_{\beta,v}) \times V(L_{\beta,v}): P_4(u,w) < K\}
\intertext{and let}
\vec Y & := V(L_{\beta,v}) \times V(L_{\beta,v}) \sm \vec X.
\end{align*}
Then, 
$$
\frac{\alpha^3}{2^8}\ell^4 \leq \frac{\alpha^3}{2^8}n^4 \leq |\P| = \sum_{(u,w) \in V(L_{\beta,v})^2} P_4(u,w)  \leq |\vec X| \cdot K + (\ell^2 -|\vec X|)\ell^2,
$$
Isolating $|\vec X|$, one arrives at 
$$
|\vec X| \leq \frac{(1-\frac{\alpha^3}{2^8})\ell^4}{\ell^2-K} = \frac{1-\frac{\alpha^3}{2^8}}{1-\frac{\alpha^3}{2^{10}}}\ell^2 \leq \Big(1- \frac{\alpha^3}{2^{10}}\Big)\ell^2.
$$
As a result, we attain
\begin{equation}\label{eq:vecY}
|\vec Y| \geq \frac{\alpha^3}{2^{10}}\ell^2 \overset{\eqref{eq:Lbv-size}}{\geq}\frac{\alpha^5}{2^{16}} n^2.
\end{equation}

In preparation for two applications of Lemma~\ref{lem:lowdegpair}, we define three graphs, namely $G_1$, $G_2$, and $G_3$, edges of which collectively capture the members of $\P$. Lemma~\ref{lem:lowdegpair} is then applied to $G_1$ and $G_3$ (along with additional parameters defined below); the resultant estimates attained from these two applications of the lemma are then used to analyse $G_2$. 

Set $G_2 := (V(L_{\beta,v}),Y)$, where $Y$ denotes the set of unordered pairs underlying $\vec Y$. For $(u,w) \in \vec Y$, set 
\begin{align*}
\vec A_{(u,w)} &:= \{(a,u): (a,u,w,b) \in \P,\; \text{for some $a,b \in V(L_{\beta,v})$}\}
\intertext{and set}
\vec B_{(u,w)}& := \{(w,b): (a,u,w,b) \in \P\; \text{for some $a,b \in V(L_{\beta,v})$}\}.
\end{align*} 
Define 
$$
G_1  := \Big(V(L_{\beta,v}), \bigcup_{(u,w) \in \vec Y} A_{(u,w)}\Big) \mand G_3 := \Big(V(L_{\beta,v}), \bigcup_{(u,w) \in \vec Y} 
B_{(u,w)}\Big),
$$
where $A_{(u,w)}$ and $B_{(u,w)}$ are the sets of unordered pairs underlying $\vec A_{(u,w)}$ and $\vec B_{(u,w)}$, respectively. 
The graphs $G_1,G_2,G_3$ are not necessarily edge disjoint.

Define the sets of vertices 
\begin{align*}
U & := \{u \in V(L_{\beta,v}): (a,u,w,b) \in \P\; \text{for some}\; a,w,b \in V(L_{\beta,v}) \mand (u,w) \in \vec Y\}\\ 
W & := \{w \in V(L_{\beta,v}): (a,u,w,b) \in \P\; \text{for some}\; a,u,b \in V(L_{\beta,v})\mand (u,w) \in \vec Y\}.
\end{align*}
Observe that $U \subseteq V(G_1)$ and that $W \subseteq V(G_3)$.

For $(u,w) \in \vec Y$, observe that $|\vec A_{(u,w)}|, |\vec B_{(u,w)}| \geq \alpha^5 n /2^{16}$. For if one of these sets, say $\vec A_{(u,w)}$, violates this inequality, then 
$$
P_4(u,w) \leq |\vec A_{(u,w)}| \deg_{L_{\beta,v}}(w) < \frac{\alpha^5}{2^{16}} n \cdot  n \overset{\eqref{eq:Ksize}}{\leq} K,
$$ 
in contradiction to $(u,w) \in \vec Y$. 
Consequently, $\deg_{G_1}(u), \deg_{G_3}(w) \geq \alpha^5 n /2^{16}$ for every $u \in U$ and every $w \in W$, respectively. 

Set 
\begin{align*}
\vec B_U &:= \{(u,w) \in U \times W: \deg_H(u,w,G_1) < d\alpha^5 n /2^{17}\},\\
\vec B_W &:= \{(u,w) \in U \times W: \deg_H(w,u,G_3) < d\alpha^5 n /2^{17}\}.
\end{align*}
Lemma~\ref{lem:lowdegpair}, applied to $G_1, U,\vec B_U$ and to $G_3, W, \vec B_W$, asserts that
\begin{equation}\label{eq:theBs}
|\vec B_U|,|\vec B_W| \leq \frac{2^{17} \rho}{d \alpha^5} n^2.
\end{equation}

Owing to~\eqref{eq:vecY}, $e(G_2) \geq \alpha^3n^2/2^{17}$ holds. This fact, together with the estimates seen in~\eqref{eq:theBs}, imply that $G_2$ admits at least 
$$
e(G_2) - |\vec B_U|-|\vec B_W| \overset{\eqref{eq:vecY}}{\geq} \frac{\alpha^5}{2^{17}} n^2 - \frac{2^{18} \rho}{d \alpha^5}  n^2 \overset{\eqref{eq:rho-cnt}}{\geq} \frac{\alpha^5}{2^{18}} n^2 
$$ 
unordered pairs $\{u,w\} \in E(G_2) \subseteq E(L_{\beta,v})$ with $u \in U$ and $w \in W$ such that $\deg_H(u,w,G_1)$, $\deg_H(w,u,G_3) \geq d\alpha^5 n/ 2^{17}$. Call these pairs in $E(G_2)$ {\em good}\footnote{As good pairs arise from edges of a simple graph these are non-degenerate.}. 

Let $(u,w) \in U \times W$ be good. At least\footnote{Strictly speaking, $\deg_H(u,w,G_1)-1$ can be replaced with $\deg_H(u,w,G_1)$ as $\{a,u,w\} \in E(H)$.} $\deg_H(u,w,G_1)-1$ neighbours $a$ of  $u$ in $G_1$ satisfy $a \not= w$. Each such neighbour $a$ of $u$ gives rise to a triple $(a,u,w)$ with the property that $au \in E(L_{\beta,v})$ so that $\deg_H(a,u) \geq \beta n$. The triple $(a,u,w)$ extends to at least\footnote{Strictly speaking, $\deg_H(w,u,G_3)-2$ can be replaced with $\deg_H(w,u,G_3)-1$ as $\{u,w,b\} \in E(H)$.} $\deg_H(w,u,G_3) - 2$ quadruples $(a,u,w,b)$ satisfying $b \notin \{a,u\}$ and $wb \in E(L_{\beta,v}))$ so that $\deg_H(w,b) \geq \beta n$ holds. Any quadruple thus formed defines a $(\beta,v)$-absorber.

It follows that for a sufficiently large $n$, a single good pair $(u,w)$ gives rise to at least 
$$
\Big(\frac{d\alpha^5}{2^{17}} n -1\Big)\Big(\frac{d\alpha^5}{2^{17}} n -2\Big) \geq \frac{d^2\alpha^{10}}{2^{35}}n^2 
$$
$(\beta,v)$- absorbers. 
Ranging over all good pairs $(u,w)$, we attain at least 
$$
\frac{\alpha^5}{2^{18}}n^2\cdot \frac{d^2\alpha^{10}}{2^{35}}n^2 = \frac{d^2\alpha^{15}}{2^{53}}n^4
$$
$(\beta,v)$-absorbers overall; concluding the proof of the lemma. 
\end{proof}

\subsection{A "small" set of $\beta$-absorbers}

Let $H$ be a $3$-graph. For $v \in V(H)$ and $\beta >0$, let $\A_{\beta,v}$ denote the set of $(\beta,v)$-absorbers in $H$.

\begin{lemma}\label{lem:F}
For every $d_{\Fref},\alpha_{\Fref}, \beta_{\Fref}, \phi_{\Fref}\in (0,1]$ such that $\beta_{\Fref} < d_{\Fref}$, there exist an integer $n_{\Fref} := n_{\Fref}(d_{\Fref},\alpha_{\Fref}, \beta_{\Fref},\phi_{\Fref})$, and reals $\rho_{\Fref}:= \rho_{\Fref}(d_{\Fref},\alpha_{\Fref},\beta_{\Fref}) >0$ and $\eta_{\Fref} := \eta_{\Fref}(d_{\Fref},\alpha_{\Fref},\phi_{\Fref}) >0 $  such that the following holds whenever $n \geq n_{\Fref}$ and $\rho < \rho_{\Fref}$. 

Let $H$ be an $n$-vertex $(\rho,d_{\Fref})_{\cherry}$-dense $3$-graph satisfying $\delta(H) \geq \alpha_{\Fref} \binom{n-1}{2}$. Then, there exists a set $\F$ of vertex-disjoint $\beta$-absorbers such that 
\begin{enumerate}
	\item [\namedlabel{prop:f1}{{\em (F.1)}}] $|\F| \leq \phi_{\Fref}n$
	\item [\namedlabel{prop:f2}{{\em (F.2)}}] For every $v \in V(H)$: $|\A_{\beta_{\Fref},v} \cap \F| \geq \eta_{\Fref}n$. 
\end{enumerate}
\end{lemma}

\begin{proof}
Given $d:=d_{\Fref}$, $\alpha:=\alpha_{\Fref}$, $\beta:= \beta_{\Fref}$, and $\phi:= \phi_{\Fref}$ as in the premise, set $\rho_{\Fref} := \rho_{\cntref}(d,\alpha,\beta)$. In addition, define the auxiliary constants
\begin{equation}\label{eq:g}
c := c_{\cntref}(d,\alpha) \mand  \gamma := \min\Big\{\frac{c}{4 \cdot 7}, \frac{\phi}{2} \Big\}.  
\end{equation}
Finally, set $\eta_{\Fref} := c\gamma/4$. 

Fix $0<\rho < \rho_{\Fref}$. Let $H$ be a $(\rho,d)_{\cherry}$-dense $3$-graph as in the premise. Then,
$
|\A_{\beta,v}| \geq c n^4
$
for every $v \in V(H)$, by Lemma~\ref{lem:cnt}. 
Let $\F'$ be a set of quadruples where each quadruple in $V(H)^4$ is put in $\F'$ independently at random with probability $\gamma n^{-3}$. Then, $\Ex\big[|\F'|\big] = \gamma n$. Chernoff's inequality~\cite[Equation~(2.9)]{JLR} then yields that 
\begin{equation}\label{eq:F'1}
|\F'| \leq 2 \gamma n \leq \phi n
\end{equation}
holds with probability $1-o(1)$. Furthermore, for every vertex $v$, 
$$
\Ex \big[|\A_{\beta,v} \cap \F'|\big] \geq cn^4 \gamma n^{-3} = c \gamma n
$$
holds. Chernoff's inequality~\cite[Equation~(2.9)]{JLR} and the union bound yield 
that 
\begin{equation}\label{eq:F'2}
|\A_{\beta,v} \cap \F'| \geq c \gamma n/2, \;\text{for every $v \in V(H)$}
\end{equation}
holds with probability $1-o(1)$.

Let $I:= I(\F')$ denote the number of pairs of members of $\F'$ meeting one another.
For a sufficiently large $n$, the total number of pairs of intersecting quadruples taken in $V(H)$ is at most 
$$
\binom{4}{1}n^4 \cdot n^3 + \binom{4}{2} (2!) n^4\cdot n^2 + \binom{4}{3}(3!)n^4\cdot n \leq 6 n^7.
$$
Then, 
$$
\Ex \big[|I|\big] \leq 6n^7 \cdot(\gamma n^{-3})^2 \leq 6\gamma^2 n.
$$
Markov's inequality now implies that 
\begin{equation}\label{eq:F'3}
|I| < 7 \gamma^2n
\end{equation}
holds with positive probability. 

It follows that an $\F'$ satisfying~\eqref{eq:F'1}, \eqref{eq:F'2}, and~\eqref{eq:F'3} exists. Fix one such $\F'$. Define $\F$ to be the set of quadruples attained from $\F'$ by, first, removing all quadruples which do not $\beta$-absorb any $v$ and, second, from each intersecting pair of quadruples remove one of the members of that pair. Property~{\em\ref{prop:f1}} trivially holds for $\F$. To see that~{\em\ref{prop:f2}} holds for $\F$, note that for every $v\in V(H)$ 
$$
|\A_{\beta,v} \cap \F| \geq c \gamma n/2 - 7 \gamma^2 n \overset{\eqref{eq:g}}{\geq} c \gamma n/4 = \eta_{\Fref}n
$$
holds whenever $n$ is sufficiently large. 
\end{proof}


\subsection{Proof of Lemma~\ref{lem:path-absorb}: $2$-degree setting}\label{sec:proof-absorb-1}

With Lemmas~\ref{lem:cnt} and~\ref{lem:F} established, we are ready to prove Lemma~\ref{lem:path-absorb}. All that remains is to "string", so to speak, the members of $\F$ (from Lemma~\ref{lem:F}) into a single path and prove the absorption capabilities of the resulting path. 
Given a quadruple $(x,y,z,w)$, we refer to $(x,y)$ and $(z,w)$ as the {\em front} and {\em rear} end-pair of the quadruple respectively.

\begin{proofof}{Lemma~\ref{lem:path-absorb}}
Let $d:=d_{\absref}$, $\alpha := \alpha_{\absref}$, $\beta:=\beta_{\absref}< \min\{\alpha,d\}$ be given. 
Set 
\begin{equation}\label{eq:kappa-mu}
\kappa:= \kappa_{\absref}:= \beta/2 \mand \mu:= \mu_{\absref} := \eta_{\Fref}(d,\alpha,\beta/20).
\end{equation}
To be clear, the definition of $\mu$ appeals to that of $\eta_{\Fref}$. The latter requires a value for $\phi_{\Fref}$ be set; here, we take $\phi_{\Fref} = \beta /20$. 
Set
\begin{equation}\label{eq:rho3}
\rho_{\absref} := \min\left\{\rho_{\Fref}(d,\alpha,\beta,\beta/20),\;\rho_{\conref}(d,\beta-\kappa) \cdot \left(1- \kappa \right)^3/2\right\}.
\end{equation}
Let $0<\rho < \rho_{\absref}$ be fixed, let $n$ be a sufficiently large integer, and let $H$ be an $n$-vertex $(\rho,d)_{\cherry}$-dense $3$-graph with $\delta_2(H) \geq \alpha(n-2)$. 

Let $\F$ denote the set of $\beta$-absorbers, existence of which in $H$ is assured by Lemma~\ref{lem:F} applied with 
$\alpha_{\Fref} = \alpha$, $d_{\Fref} = d$, $\beta_{\Fref} = \beta$, $\phi_{\Fref} = \beta/20$, and owing to $\rho < \rho_{\Fref}(d,\alpha,\beta,\beta/20)$, by~\eqref{eq:rho3}. 
Fix an arbitrary ordering of the members of $\F$, namely
$F_1,F_2,\ldots,F_r$, where $r := |\F| \leq \beta n/20$, by~{\em\ref{prop:f1}}. 
In what follows, we prove that a path $A$ of the form
\begin{equation}\label{eq:path-A}
F_1 \circ P_1 \circ \cdots \circ F_{r-1} \circ P_{r-1} \circ F_r
\end{equation}
exists in $H$, where here each $P_i$ is a $10$-path connecting the rear end-pair of $F_i$ with the front end-pair of $F_{i+1}$; we use $\circ$ to denote path concatenations along these pairs. If such a path $A$ were to exist, then it would form a $(\beta,\mu,\kappa)$-absorbing path. To see this, observe, first, that  
\begin{equation}\label{eq:path-A-length}
|V(A)| = 4 r +6(r-1)  \leq 10 r \leq  \beta n /2  \overset{\eqref{eq:kappa-mu}}{=} \kappa n. 
\end{equation}
Observe, second, that~{\em\ref{prop:f2}} together with a standard greedy argument (see, e.g.,~\cite[Claim~2.6]{RRS06}), assert that such a path $A$ would form a $\mu n$-absorbing path. Observe, third, that the ends of such a path $A$ would have codegree at least $\beta n$ for, indeed, $\beta < \alpha$, and $\delta_2(H) \geq \alpha(n-2)$ (here we utilise the fact that $n$ is sufficiently large). 

It remains to establish the existence of the aforementioned path. This we do inductively as follows. Put $A_1 := F_1$. Suppose that the (partial) path 
\begin{equation}\label{eq:Ai}
A_i := F_1 \circ P_1 \circ \cdots \circ F_{i-1} \circ P_{i-1} \circ F_i
\end{equation}
has been defined for some $i \in [r-1]$. We define $A_{i+1}$ as follows. 
Set 
\begin{equation}\label{eq:Vi}
V_i := \bigg(V(H) \sm \big(V(A_i) \cup V(\F)\big)\bigg) \cup \{a,b,c,d\},
\end{equation}
where $(a,b)$ is the rear end-pair of $F_i$ and $(c,d)$ is the front end-pair of $F_{i+1}$. The next two claims verify that Lemma~\ref{lem:connect} can be applied to $H[V_i]$ in order to connect $(a,b)$ with $(c,d)$ via $H[V_i]$. 

\begin{claim}
Let $i \in [r]$. Then, $H[V_i]$ is $(\rho',d)_{\cherry}$-dense for some $\rho'< \rho_{\conref}(d,\beta/2)$.
\end{claim}

\begin{innerproof}
As $|V(A_i) \cup V(\F)| \leq \kappa n$, it follows that $|V_i| \geq \left(1- \kappa\right) n$ for every $i \in [r]$. Fix $\vec G_1, \vec G_2 \subseteq V_i \times V_i$, and note that 
\begin{align}
e_{H[V_i]}(\vec G_1,\vec G_2) & \overset{\phantom{\eqref{eq:rho3}}}{\geq} d|\P_2(\vec G_1,\vec G_2)| - \rho n^3 \nonumber \\
& \overset{\eqref{eq:rho3}}{\geq} d|\P_2(\vec G_1,\vec G_2)| - \rho_{\conref}(d,\beta-\kappa)\frac{(1-\kappa)^3}{2} n^3 \label{eq:middle}\\  
& \overset{\phantom{\eqref{eq:rho3}}}{\geq} d|\P_2(\vec G_1,\vec G_2)| - \frac{\rho_{\conref}(d,\beta-\kappa)}{2}|V_i|^3. \nonumber
\end{align}
Owing to $\kappa = \beta/2$ the claim follows. 
\end{innerproof}

\begin{claim}
Let $i \in [r]$. Then, $\delta_2(H[V_i]) \geq \beta n /2$. 
\end{claim}

\begin{innerproof}
Owing to $\beta < \alpha$,  $\delta_2(H) \geq \alpha(n-2)$, and $n$ being sufficiently large, we may write 
$$
\delta_2(H[V_i]) \geq \alpha (n-2) -\kappa n = (\alpha -\kappa)n -2 \alpha \overset{\alpha > \beta}{\geq} (\beta - \kappa)n \overset{\eqref{eq:kappa-mu}}{=} \beta n /2.
$$
\end{innerproof}

Lemma~\ref{lem:connect}, applied to $H[V_i]$ with $\alpha_{\conref} = \beta/2$ and $d_{\conref} = d$, asserts that any two pairs of vertices in $H[V_i]$ can be connected via a $10$-path in $H[V_i]$. This in particular holds for the pairs $(a,b)$ and $(c,d)$. The path $P_{i+1}$, as defined above, exists and consequently $A_{i+1}$ as well. This completes the proof of the existence of $A$ and thus concludes the proof of the lemma.  
\end{proofof}

\subsection{Proof of Lemma~\ref{lem:path-absorb-2}: $1$-degree setting}\label{sec:proof-absorb-2}

%
Given $\alpha := \alpha_{\absrefone}$, $d := d_{\absrefone}$, and $\eta:=\eta_{\absrefone}$ as in the premise of the lemma, choose two constants  $0 <\zeta \ll \kappa$ such that 
\begin{equation}\label{eq:zeta-kappa}
\kappa \ll \min\bigg\{\alpha,d,\eta,\frac{\alpha + d -1 }{2(d+1)+1}\bigg\},\; \zeta + \kappa \ll g(\alpha-\kappa,d)
\end{equation}
holds, where here $g(\cdot)$ is as defined in~\eqref{eq:g-func}. 
The first constraint being our prerogative, we explain the validity of the second. Owing to $\zeta \ll \kappa$, it suffices to argue that a choice for $\kappa$ satisfying 
$$
2 \kappa \ll g(\alpha- \kappa,d) \overset{\eqref{eq:g-func}}{=} \min\Big\{\alpha -\kappa,d,\frac{\alpha -\kappa + d - 1}{d+1} \Big\}
$$ 
exists. The first term in the minimisation entails having to require $\kappa \ll \alpha$. The second term imposes $\kappa \ll d$. The third, and final, term requires
$
\kappa \ll \frac{\alpha + d -1 }{2(d+1)+1}
$.

We remark that the condition $\kappa \ll \eta$ plays no role in the proof. It is, however, mandated in order to accommodate a subsequent application of Lemma~\ref{lem:path-absorb-2} in the proof of Theorem~\ref{thm:main2} in~\S~\ref{sec:proof-main-2}. 

With $\zeta$ and $\kappa$ fixed, define 
$$
\kappa_{\absrefone}:= \kappa \mand \mu:=\mu_{\absrefone} := \eta_{\Fref}(d,\alpha,\kappa/10).
$$
To be clear, the definition of $\mu$ entails setting $\phi_{\Fref} := \kappa/10$. Set
\begin{equation}\label{eq:rho3-2}
\rho_{\absrefone} := \min\{\rho_{\Fref}(\alpha,d,d-\zeta,\kappa/10),\; 2^{-1}(1-\kappa)^3\rho_{\conrefone}(d,\alpha-\kappa,\zeta+2\kappa)\}.
\end{equation}

Let $0<\rho < \rho_{\absrefone}$ be fixed and let $H$ be a $(\rho,d)_{\cherry}$-dense $3$-graph satisfying $\delta(H) \geq \alpha_{\absrefone}\binom{n-1}{2}$ be given. Let $\F$ be a set of $(d- \zeta)$-absorbers, existence of which is assured by Lemma~\ref{lem:F} applied with 
$\alpha_{\Fref} = \alpha$, $d_{\Fref} = d$, $\beta_{\Fref} = d-\zeta$, and $\phi_{\Fref} = \kappa/10$; and also owing to $\rho < \rho_{\Fref}(\alpha,d,d-\zeta,\kappa/10)$ per~\eqref{eq:rho3-2}. As in the proof of Lemma~\ref{lem:path-absorb}, we seek to establish the existence of a path $A$ of the form~\eqref{eq:Ai}. If such a path $A$ were to exist, then it would form a $(\beta-\kappa, \mu, \kappa)$-absorbing path. Indeed, owing to $r:=|\F| \leq \kappa n /10$ its length would be at most $\kappa n$, by~\eqref{eq:path-A-length}; it would be $\mu n$-absorbing, by~\cite[Claim~2.6]{RRS06}; and its ends would have codegree at least $(d-\eta)n$ as these arise from $(d-\zeta)$-absorbers in $\F$ and $\zeta \ll \eta$.

It remains to establish the existence of $A$. Let $A_i$ and $V_i$ be as defined in~\eqref{eq:Ai} and~\eqref{eq:Vi}, respectively. 
Suffice to prove that the pairs $(a,b)$ and $(c,d)$ (per the definition of $V_i$) can be connected via a $10$-path passing through $H[V_i]$. This we accomplish using Lemma~\ref{lem:con-1deg}. Hence, it remains to prove that for every $i \in [r]$, $H[V_i]$ adheres to the premise of that lemma. The following claims verify this.

Starting with the $\cherry\;$-denseness of $H[V_i]$, note that~\eqref{eq:rho3-2}, the observation that $|V_i| \geq (1-\kappa)n$, and an argument identical to that seen in~\eqref{eq:middle} establish the following. 

\begin{claim}
Let $i \in [r]$. Then, $H[V_i]$ is $(\rho',d)_{\cherry}$-dense for some $\rho' <	\rho_{\conrefone}(d,\alpha-\kappa,\zeta+2\kappa)$. 
\end{claim}

\begin{claim}
Let $i \in [r]$. Then, $\delta(H[V_i]) \geq (\alpha - \kappa)\binom{|V_i|-1}{2}$.
\end{claim}

\begin{innerproof}
Start by observing that 
$$
	\delta(H[V_i]) \geq \delta(H) - |V(A_i) \cup V(\F)|\cdot n,
$$
where the last term on the r.h.s. accounts for all pairs involving vertices from $V(A_i) \cup V(\F)$. Owing to~\eqref{eq:path-A-length} and relaying on $n$ being sufficiently large, we may write
\begin{align} 
	\delta(H[V_i]) & \geq \alpha \binom{n-1}{2} - \kappa n^2 
	 \geq \alpha \frac{n^2}{2} - \kappa n^2 - \alpha \frac{n}{2} \nonumber\\
	& = \big(\alpha -\kappa/2\big)\frac{n^2}{2} -\alpha \frac{n}{2}  
	 \geq \big(\alpha -\kappa\big)\frac{n^2}{2} \label{eq:been-here}\\
	& \geq \big(\alpha -\kappa\big) \binom{|V_i|-1}{2}. \nonumber 
	\end{align}
\end{innerproof}


Having set $\kappa \ll \alpha +d - 1$ in~\eqref{eq:zeta-kappa}, implies that $(\alpha -\kappa) + d > 1$ holds. This, coupled with the condition $\zeta + \kappa \ll g(\alpha-\kappa,d)$, also set in~\eqref{eq:zeta-kappa}, imply that any two disjoint pairs of vertices having codegree at least $(d-\zeta - \kappa)n$ can be connected in $H[V_i]$ via a $10$-path, by Lemma~\ref{lem:con-1deg}. In $H[V_i]$, both pairs $(a,b)$ and $(c,d)$ have codegree at least $(d-\zeta-\kappa)n$ and thus connectable in this fashion. The existence of $A$ is established. 
This concludes the proof of the lemma. 

\section{The path-cover lemma}

In this section we prove our path-cover lemma, namely Lemma~\ref{lem:path-cover}. 
Our proof of this lemma employs the weak regularity lemma stated below in Lemma~\ref{lem:reg}. In~\S~\ref{sec:pc-no-reg-points}, we provide an alternative proof of Lemma~\ref{lem:path-cover} for graphs equipped with the notion of $\points\;$-denseness; the latter notion is a stronger notion than that of $1$-set-denseness assumed in Lemma~\ref{lem:path-cover}. If $\points\;$-denseness is assumed, then the regularity lemma is no longer needed giving rise to a much shorter proof.  

\subsection{Path covers in $1$-set-dense $3$-graphs}

A $3$-graph $H$ is said to be $t$-{\em partite} if there is a vertex partition $V(H) = V_1 \discup V_2 \discup \cdots \discup V_t$ such that each edge $e \in E(H)$ satisfies $|e \cap V_i| \leq 1$ whenever $i \in [t]$. A $t$-partite $H$ is said to be {\em equitable} if its $t$-partition satisfies 
$
|V_1| \leq |V_2| \leq \cdots \leq |V_t| \leq |V_1| +1
$. 
We also refer to the partition itself as {\sl equitable}. An $n$-vertex $3$-partite $3$-graph $H$ with an underlying partition $V(H) = X \discup Y \discup Z$ is said to be $\eps$-{\em regular} if 
\begin{equation}\label{eq:regular}
e_H(X',Y',Z') = \frac{e_H(X,Y,Z)}{|X||Y||Z|}|X'||Y'||Z'| \pm \eps n^3 
\end{equation} 
holds for every $X' \subseteq X$, $Y' \subseteq Y$, and $Z' \subseteq Z$. If only the lower bound seen at~\eqref{eq:regular} is is upheld by $H$, then we refer to such an $H$ as $\eps$-{\em lower-regular}. If in addition $e_H(X,Y,Z) / |X||Y||Z| \geq d$, then $H$ is called $(\eps,d)$-{\em regular} or $(\eps,d)$-{\em lower-regular}, respectively.
The following result is a commonly known generalisation of the main result of~\cite{Szemeredi}. 

\begin{lemma}\label{lem:reg}{\em{\bf (Weak-regularity lemma for $3$-graphs}~\cite{Szemeredi}{\bf)}}
For every $\eps_{\regref} >0$ and integer $t_{\regref}$ there exist integers $n_{\regref}$ and $T_{\regref}$ such that the following holds whenever $n \geq n_{\regref}$. 

Let $H$ be an $n$-vertex $3$-graph. Then, there exists an integer $t$ satisfying $t_{\regref} \leq t \leq T_{\regref}$ and an equitable partition $V(H) = V_1 \discup V_2 \discup \cdots \discup V_t$ such that 
for all but at most $\eps t^3$ triples $i,j,k \in [t]$ the sets $V_i,V_j,V_k$ induce an $\eps_{\regref}$-regular $3$-partite $3$-graph denoted $H[V_i,V_j,V_k]$. 
\end{lemma}

Given a $3$-graph $H$, regularised per Lemma~\ref{lem:reg}, and a real $d >0$, define $R_d := R_d(H)$ to denote the $3$-graph whose vertices are the {\em clusters} (i.e., sets) $(V_i)_{i \in [t]}$ and whose edges are the triples $\{V_i,V_j,V_k\}$, $i,j,k \in [t]$, such that $H[V_i,V_j,V_k]$ is $(\eps,d)$-regular. It will be convinient to identify $V(R_{d})$ with $[t]:=\{1,\ldots,t\}$. Given $X \subseteq V(R_d)$ define $\cup X := \bigcup_{i \in X} V_i$. An edge $e \in E(H)$ is said to be {\em crossing with respect to $X$} if there are three clusters $V_i,V_j,V_k$ {\sl captured} by $X$ such that 
$$
1 = |e\cap V_i| = |e\cap V_j|  = |e\cap V_k|.
$$

\begin{lemma}\label{lem:pack}{\em{\bf (Path packing lemma} \cite[Claim~4.2]{RRS08}{\bf)}}
For all $0<\eps < d < 1$, every $(\eps,d)$-lower-regular $3$-partite equitable $3$-graph $H$ on $n$ vertices, with $n$ a sufficiently large integer, contains a family $\P$ of vertex disjoint-paths such that for each $P \in \P$ we have 
$$
|V(P)| \geq \eps(d-\eps)n/3 \mand \sum_{P \in \P} |V(P)| \geq (1-2\eps)n.
$$  
\end{lemma}

The following is a triviality whose proof is included for completeness.  

\begin{lemma}\label{lem:matching} 
For every $d_{\matref} > 0$ and $\zeta_{\matref} > 0$, there exist an integer $n_{\matref}:=n_{\matref}(d_{\matref},\zeta_{\matref}) >0 $ and a real $\rho_{\matref}(d_{\matref},\zeta_{\matref}) >0 $ such that the following holds whenever $n \geq n_{\matref}$ and $0 < \rho < \rho_{\matref}$. 

Let $H$ be an $n$-vertex $(\rho,d_{\matref})$-dense $3$-graph. Then, $H$ admits a matching covering all but at most $\max\{2,\zeta_{\matref}n\}$ vertices. 
\end{lemma}

\begin{proof}
Given $d:= d_{\matref}$ and $\zeta := \zeta_{\matref}$, set $
\rho_{\matref} := \frac{d \cdot \zeta^3}{27}$. Let $0< \rho < \rho_{\matref}$, let $n$ be sufficiently large, and let $H$ be an $n$-vertex $3$-graph as in the premise. Let $M$ be a maximum matching in $H$. Let $X := V(H) \sm V(M)$ denote the set of vertices not covered by the members of $M$. If $|X| \leq 2$, then the claim follows. Assume then that $|X| \geq 3$. In which case, $e_H(X) = 0$ by the maximality of $M$. Then, 
$$
d\binom{|X|}{3} - \rho n^3 \leq e_H(X) \leq 0.
$$
Consequently
$$
\frac{d}{27}|X|^3 \leq \rho n^3.
$$
Assuming that $|X| > \zeta n$ we arrive at 
$$
\zeta n < |X| \leq (27 \cdot \rho d^{-1})^{1/3} n
$$
contradicting $\rho < \rho_{\matref}$ set at the outset. Consequently, in this case, $|X| \leq \zeta n$ must hold. 
\end{proof}

We are now ready to prove our path-cover lemma, namely Lemma~\ref{lem:path-cover}.

\begin{proofof}{Lemma~\ref{lem:path-cover}}
Given $d:= d_{\pcref}$ and $\zeta := \zeta_{\pcref}$ let $\rho':= \rho_{\matref}(d/2,\zeta/12)$ and set 
\begin{equation}\label{eq:reg-cons}
t_{\reg}:= \max\{8/\rho',8/\zeta\}, \quad d' := \rho'/4, \quad \eps_{\reg} := \min\{d'/2,\zeta/24\}.
\end{equation}
In addition, set 
\begin{equation}\label{eq:cons-local}
\rho_{\pcref} := \rho'/4\; \mand \; \ell_{\pcref} := \frac{T_{\regref}(\eps_{\reg},t_{\reg})}{\eps_{\reg}(d'-\eps_{\reg})}.
\end{equation}
Let $n$ be sufficiently, let $\rho < \rho_{\pcref}$, and let $H$ be an $n$-vertex $(\rho,d)$-dense $3$-graph. 

Let $R_{d'} :=R_{d'}(H)$ denote the reduced graph of $H$ obtained after regularising $H$  using the weak-regularity lemma, namely Lemma~\ref{lem:reg}, applied with $\eps_{\regref}=\eps_{\reg}$ and $t_{\regref} = t_{\reg}$. Let $t:=|V(R_{d'})|$ and identify $V(R_{d'})$ with $[t]$.

\begin{claim}\label{clm:rd-dense}
\text{$R_{d'}$ is $(\rho',d/2)$-dense.}
\end{claim}

\begin{innerproof} Fix $X \subseteq V(R_{d'})$ and let $C_X$ denote the number of edges of $H$ which are crossing with respect to $X$ and that lie in $(\eps_{\reg},d')$-regular triples $H[V_i, V_j, V_k]$, where the sets $V_i,V_j,V_k$ are taken from the underlying regularity partition. Then, 
$$
e_{R_{d'}}(X) \geq C_X/2(n/t)^3;
$$ 
the factor $2$ appearing here is incurred in order to cope with the the fact that cluster sizes are in the set $\{n/t,n/t+1\}$; we use the fact that for a sufficiently large $n$, $2(n/t)^3 \geq (n/t+1)^3$ holds. 

Observe that 
\begin{align*}
C_X \geq e_H(\cup X) - |X| \cdot 2(n/t)^3 - |X|^2\cdot  2(n/t)^3 -\eps_{\reg} t^3 \cdot 2(n/t)^3-|X|^3 d' \cdot 2(n/t)^3.
\end{align*}
Indeed, the second and third terms on the r.h.s. arise from the removal of all edges that have at least two of their vertices in the same cluster captured by $X$ from $E(H[\cup X])$. The fourth term on the r.h.s. arises due the removal of  all (crossing) edges found in $\eps_{\reg}$-irregular triples of clusters. Finally, the last term on the r.h.s. arises from the removal of all (crossing) edges found in triples of clusters whose edge density is at most $d'$. 

As $|X| \leq t$, we arrive at 
\begin{align*}
e_{R_{d'}}(X) \geq \frac{e_H(\cup X)}{2(n/t)^3} - t - t^2 - \eps_{\reg}t^3 -d' t^3.
\end{align*}
As $H$ is $(\rho,d)$-dense, 
$$
e_H(\cup X) \geq d \binom{\sum_{i \in X} |V_i|}{3} - \rho n^3 \geq d\binom{|X|}{3} (n/t)^3 - \rho n^3
$$
holds. Indeed, the term $\binom{|X|}{3} (n/t)^3$ accounts only for edges crossing with respect to $X$ while $\binom{\sum_{i \in X} |V_i|}{3}$ accounts also for triples inside clusters captured by $X$.
By~\eqref{eq:reg-cons}, $t+t^2 \leq 2 t^2 \leq \rho' t^3 /4$ holds. By~\eqref{eq:cons-local}, $\rho \leq \rho'/4$ holds. We may now write   
$$
e_{R_{d'}}(X) \geq \frac{d}{2} \binom{|X|}{3} - (\rho +\rho'/4 + \eps_{\reg} + d')t^3 \geq \frac{d}{2} \binom{|X|}{3} - \rho' t^3,
$$
and the claim follows.
\end{innerproof}

In view of Claim~\ref{clm:rd-dense} and the choice of $\rho'$, it follows, by Lemma~\ref{lem:matching}, that $R_{d'}$ admits a matching $M$ missing at most $\max\{2,\zeta t /12\}$ vertices of $R_{d'}$. For each edge $(V_i,V_j,V_k)$ of $M$, apply Lemma~\ref{lem:pack} to $H[V_i,V_j,V_k]$ as to obtain a system of vertex-disjoint paths as described in Lemma~\ref{lem:pack}. Let $\P$ denote the system of paths thus generated in $H$ over all edges of $M$.  
In each $H[V_i,V_j,V_k]$ corresponding to an edge $(V_i,V_j,V_k$) of $M$, at most $\frac{3}{\eps_{\reg}(d'-\eps_{\reg})}$ paths are {\sl packed}. As $|M| \leq T_{\regref}(\eps_{\reg},t_{\reg}) / 3$ at most $\ell_{\pcref}$ paths are thus packed. 

It remains to argue that the members of $\P$ cover all but at most $\zeta n$ vertices of $H$. In each $H[V_i,V_j,V_k]$ corresponding to an edge $(V_i,V_j,V_k$) of $M$ at most $2 \eps_{\reg}\cdot 6 n/t$ vertices of $H[V_j,V_j,V_k]$ are missed. As $|M| \leq t /3$, at most $12\eps_{\reg} n$ vertices of $H$ are missed this way. From the clusters not covered by $M$ at most $\max\{2,\zeta t/12\} \cdot 2 n/t $ vertices of $H$ are missed. Overall at most $(12 \eps_{\reg}+ \max\{4 /t,\zeta/2\})n $ vertices of $H$ are missed. Owing to~\eqref{eq:reg-cons}, $12 \eps_{\reg} \leq \zeta/2$ and $t\geq t_{\reg} \geq 8/ \zeta$ (so that $4 /t \leq \zeta/2$); consequently 
$12 \eps_{\reg}+ \max\{12/t,\zeta/2\} \leq \zeta$ as required.
\end{proofof}

\subsection{Path-covers in $3$-set-dense $3$-graphs}\label{sec:pc-no-reg-points}

In this section, we provide a significantly shorter proof for Lemma~\ref{lem:path-cover}, under the {\sl strengthened} assumption that the host $3$-graph is $\points\;$-dense and not merely $1$-set-dense. 
%
The main ingredient, so to speak, of the argument for constructing path-covers in $\points\;$-dense graphs is the following. 

\begin{lemma}\label{lem:pack-2} {\em \cite[Claim~4.1]{RRS08}} 
Let $c > 0$. Then, every $3$-partite $3$-graph $H$ having at most $m$ vertices in each partition set and satisfying $e(H) \geq cm^3$, contains a path on at least  $cm$ vertices. 
\end{lemma}

\begin{proofof}{Lemma~\ref{lem:path-cover} for $\points\;$-dense $3$-graphs}
Given $d:=d_{\pcref}$ and $\zeta:= \zeta_{\pcref}$, set $\rho_{\pcref} := \frac{d\zeta^3}{2 \cdot 3^3}$, and set $\ell_{\pcref} := \ciel{1/\rho_{\pcref}}$. Let $0 < \rho < \rho_{\pcref}$ be fixed, let $n$ be sufficiently large, and let $H$ be a $(\rho,d)_{\points}$-dense $n$-vertex $3$-graph. 

We define a sequence of subgraphs $H_0:=H \supseteq H_1 \supseteq H_2 \cdots$ as follows. Let $H_i$, for some $i \geq 0$, be given. If $n_i := |V(H_i)| < \zeta n$, then set $H_{i+1}$ to be the empty graph. Otherwise, note that $e_{H_i}(V_1,V_2,V_3) = e_H(V_1,V_2,V_3)$ whenever $V_1,V_2,V_3 \subseteq V(H_i)$. 
Choose an arbitrary equipartition $V(H_i) = U_1 \discup U_2 \discup U_3$ such that 
$|U_1| \leq |U_2| \leq |U_3| \leq |U_1|+1$. Then, $|U_i| \in \{n_i/3, n_i/3+1\}$, for every $i \in [3]$. The $3$-partite subgraph of $H_i$ induced by the edges of $H_i$ crossing $U_1,U_2$, and $U_3$ has $e_{H_i}(U_1,U_2,U_3)$ edges. For the latter quantity we observe that 
$$
e_{H_i}(U_1,U_2,U_3) = e_H(U_1,U_2,U_3) \geq d |U_1||U_2||U_3| -\rho n^3 \geq \left(\frac{d\zeta^3}{3^3}-\rho \right)n^3 \geq \frac{d\zeta^3}{2 \cdot 3^3}n^3.
$$  
As $|U_j| \leq 2n/3$ holds for every $j\in [3]$, it follows, by Lemma~\ref{lem:pack-2}, that $H_i$ contains a path $P_i$ of length at least $\frac{d\zeta^3}{3^4}n$. Set $H_{i+1} := H_i - V(P_i)$. 

In the above sequence, all graphs $H_i$ with $i > \ell_{\pcref}$ are empty. Hence, At most $\ell_{\pcref}$ paths are defined throughout the above process and these form the required path-cover of $H$. 
\end{proofof}

\section{Proofs of the main results}\label{sec:proof}

Let $H$ be a $3$-graph, let $x,y \in V(H)$, and let $H' \subseteq H$ with $\{x,y\} \not\subseteq V(H')$ possible. 
Define 
$$
\deg_{H'}(x,y) : = |N_H(x,y) \cap V(H')|. 
$$


\subsection{Proof of Theorem~\ref{thm:main}: $2$-degree setting}\label{sec:proof-main-1}

Given $d> 0$ and $\alpha >0$, set 
\begin{equation}\label{eq:kmg}
0 < \beta < \min\{d,\alpha\},\; \kappa := \kappa_{\absref}(d,\alpha,\beta)\leq \beta/2,\; \mu := \mu_{\absref}(d,\alpha,\beta),\; \nu \ll \mu,\; \zeta:= \mu - 2\nu. 
\end{equation}
The inequality $\kappa \leq \beta/2$ is supported by Lemma~\ref{lem:path-absorb}. 
Set 
\begin{equation}\label{eq:rho-final}
0< \rho < \min \left\{\rho_{\absref}(d,\alpha,\beta),2^{-1}(1-\kappa-\nu)^3  \cdot \rho_{\pcref}(d,\zeta), \rho_{\conref}(d,\beta/8)\cdot(\nu(1-\kappa)/4)^3\right\}.
\end{equation}

Let $n$ be sufficiently large and let $H$ be an $n$-vertex $(\rho,d)_{\cherry}$-dense $3$-graph satisfying $\delta_2(H) \geq \alpha (n-2)$. 
As $\rho < \rho_{\absref}(d,\alpha,\beta)$, $H$ admits a $(\beta,\mu,\kappa)$-absorbing path $A$, by Lemma~\ref{lem:path-absorb}. Define $H' := H -V(A)$ (i.e., $H'$ is attained from $H$ by removing the vertices of $A$ from the latter). 

We prepare for an application of Lemma~\ref{lem:res-2}. For each pair of vertices $\{x,y\} \in \binom{V(H)}{2}$, set $U_{\{x,y\}} : = N_H(x,y) \cap V(H')$. Then, 
$$
\deg_{H'}(x,y) = |U_{\{x,y\}}| \geq \alpha (n-2) - |V(A)| \geq \alpha (n-2) - \kappa n \geq (\alpha -\kappa)n -2\alpha \overset{\alpha > \beta}{\geq} (\beta-\kappa)n \overset{\eqref{eq:kmg}}{\geq} \beta n/2,
$$
holds for every $\{x,y\}\in \binom{V(H)}{2}$. Lemma~\ref{lem:res-2}, applied with $\nu_{\resrefone}=\nu$, then asserts that there exists a set $R \subseteq V(H')$ satisfying 
\begin{equation}\label{eq:R-codegree}
|N_H(x,y) \cap R| \geq (\beta/2 -2n^{-1/3})|R| \geq \beta|R|/4
\end{equation}
for every $\{x,y\} \in \binom{V(H)}{2}$. Consequently, $\delta_2(H[R]) \geq \beta |R| /4$. Moreover, the set $R$ also satisfies $|R| = \nu n' \pm \nu (n')^{2/3}$, where $n' := |V(H')| \geq (1-\kappa)n$. One may then write 
\begin{equation}\label{eq:R-size}
\frac{\nu(1-\kappa)}{2}n \leq |R| \leq 2 \nu n.
\end{equation}

Set $H'':= H' -R$ (i.e., $H''$ is attained from $H'$ by removing the members of $R$). Then, 
\begin{equation}\label{eq:n2}
n'' := |V(H'')| \geq n -|V(A)| - |R| \overset{\eqref{eq:R-size}}{\geq} n- \kappa n - 2 \nu n = (1- \kappa-2 \nu)n. 
\end{equation}

\begin{claim}\label{clm:H_2-dense}
$H''$ is $(\xi_1,d)_{\cherry}$-dense for some $\xi_1 < \rho_{\pcref}(d,\zeta)$. 
\end{claim}

\begin{innerproof}
Fix $\vec G_1,\vec G_2 \subseteq V(H'') \times V(H'')$ and note that as $H''$ is an induced subgraph of $H$, then 
\begin{align*}
e_{H''}(\vec G_1,\vec G_2) &\overset{\phantom{\eqref{eq:rho-final}}}{\geq} d|\P_2(\vec G_1,\vec G_2)| -\rho n^3\\
& \overset{\eqref{eq:rho-final}}{\geq} d|\P_2(\vec G_1,\vec G_2)| - 2^{-1} \cdot \rho_{\pcref}(d,\zeta)(1-\kappa-2\nu)^3 n^3 \\
 &\overset{\eqref{eq:n2}}{\geq} d|\P_2(\vec G_1,\vec G_2)| - 2^{-1} \cdot \rho_{\pcref}(d,\zeta)(n'')^3; 
\end{align*}
\end{innerproof}

As $\cherry$-denseness implies $1$-set-denseness, it follows that $H''$ is $(\xi_1,d)$-dense for some $\xi_1 < \rho_{\pcref}(d,\zeta)$. Lemma~\ref{lem:path-cover} then asserts that $H''$ admits a collection $\P' := \{P_1,\ldots,P_{h-1}\}$, $h-1 \leq \ell_{\pcref}(d,\zeta)$, of vertex-disjoint paths covering all but at most $\zeta n'' \leq \zeta n$ vertices of $H''$ and thus of $H$ as well. Write $P_h := A$ and set $\P := \P' \cup \{P_h\}$. 

In what follows we use the set $R$ in order to concatenate the members of $\P$ into a (tight) cycle. This entails $h$ applications of the connecting lemma fitting to this setting, namely Lemma~\ref{lem:connect}. We proceed in two steps. First, a path 
of the form 
\begin{equation}\label{eq:L-path}
L := P_1 \circ K_1\circ P_2 \circ K_2 \circ\cdots \circ K_{h-1} \circ P_h
\end{equation}
is constructed, where here each $K_i$ is a $10$-path disjoint of all other $10$-paths involved in the construction. 
Second, the remaining "free" end-pair of $P_h$ is connected using an additional $10$-path with the remaining "free" end-pair of $P_1$. Also here we 
The resulting cycle we denote by $C$. We now make this precise.    

The construction of $L$ is done inductively. Set $L_1 := P_1$. Assuming 
\begin{equation}\label{eq:Li-partial}
L_i:=P_1 \circ K_2 \cdots\circ K_{i-1}\circ P_i
\end{equation}
has been constructed for some $i \in [h-1]$, we define $L_{i+1}$ as follows. Let $\{a,b\}$ be the free end-pair of $P_i$ and let $\{c,d\}$ be one of the end pairs of $P_{i+1}$. Set 
\begin{equation}\label{eq:Ri-partial}
R_i := (R \sm V(L_i)) \cup \{a,b,c,d\}.
\end{equation}
Observing that $|V(L_i) \cap R| \leq 10 h$, we may write that 
\begin{equation}\label{eq:Ri-size}
|R_i| \geq |R| - 10 h \overset{\eqref{eq:R-size}}{\geq} \nu(1-\kappa)n/2 -10h \geq \nu(1-\kappa)n/4
\end{equation}
holds for a sufficiently large $n$. 
Owing to $\rho < \rho_{\conref}(d,\beta/8) \cdot (\nu(1-\kappa)/4)^3$, by~\eqref{eq:rho-final}, an argument identical to that seen in Claim~\ref{clm:H_2-dense} establishes that $H[R_i]$ is $(\xi_2,d)_{\cherry}$-dense for some $\xi_2 < \rho_{\conref}(d,\beta/4)$. Moreover, 
$$
\delta_2(H[R_i]) \overset{\eqref{eq:R-codegree}}{\geq} \beta|R|/4 - 10h \geq \beta|R| /8 \geq \beta |R_i| /8,
$$
holds whenever $n$ is sufficiently large. The path $K_i$ connecting $\{a,b\}$ and $\{c,d\}$ in $H[R_i]$ then exists, by Lemma~\ref{lem:connect} applied with $d_{\conref} =d $ and $\alpha_{\conref} = \beta/8$. This completes the construction of $L$. Completing $L$ into the aforementioned cycle $C$ is done using the exact same argument provided for $K_i$. 

With respect to the original $3$-graph $H$, the cycle $C$ covers all vertices of the latter but (a subset of) those found in $V(H'') \sm V(\P)$ and those vertices of $R$ not used throughout the construction of $C$. The number of uncovered vertices is then at most 
$$
|V(H'') \sm V(\P)| + |R| \leq \zeta n + 2 \nu n \overset{\eqref{eq:kmg}}{=} (\mu - 2 \nu)n +2 \nu n = \mu n.
$$
The path $A$, present in $C$, can then absorb all uncovered vertices, rendering a tight Hamilton cycle in $H$. This completes the proof of Theorem~\ref{thm:main}.

\subsection{Proof of Theorem~\ref{thm:main2}: $1$-degree setting}\label{sec:proof-main-2}

Given $\alpha > 0 $ and $d >0$ satisfying $\alpha +d > 1$, set 
$$
\eta \ll \min\bigg\{g(\alpha,d),\; a+d-1,\; \frac{\alpha +d-1}{3(d+1)+2}\bigg\}.
$$ 
Recalling the definition of $g(\cdot)$ from~\eqref{eq:g-func}, note that we have just imposed $\eta \ll \alpha, d$ as well. 
With $\eta$ set, define 
$$
\kappa:= \kappa_{\absrefone}(d,\alpha,\eta) \ll \eta,\; \mu := \mu_{\absrefone} (d,\alpha,\eta),\; \nu \ll \mu,\; \zeta:= \mu -2 \nu,
$$
where the inequality $\kappa \ll \eta$ is supported by Lemma~\ref{lem:path-absorb-2}. 
As $\kappa \ll \eta \ll \alpha + d - 1$, we may insist on $\alpha -\kappa +d -1 >0$. 
Fix an auxiliary constant $0 < \gamma \ll \eta$ such that 
\begin{equation}\label{eq:legal-1}
\alpha -\kappa - \gamma + d -1 >0.
\end{equation}
Define an additional auxiliary constant $0 < \gamma' \ll \eta$ such that 
\begin{equation}\label{eq:legal-2}
\eta + \kappa + \gamma' \ll g(\alpha -\kappa-\gamma,d) \overset{\eqref{eq:g-func}}{=} \min \bigg\{\alpha -\kappa-\gamma,d, \frac{\alpha-\kappa-\gamma+d-1}{d+1} \bigg\}.
\end{equation}
To see that~\eqref{eq:legal-2} is possible, we iterate over each of the constraints involved in it. For the first one, note that the requirement $\eta + \kappa + \gamma' \ll \alpha -\kappa-\gamma$ can be rewritten as to read 
$\eta +2\kappa +\gamma + \gamma' \ll \alpha$. Owing to being able to choose $\kappa,\gamma,\gamma' \ll \eta$ this amounts to requiring $ \eta \ll \alpha$; the latter is already imposed on $\eta$. In a similar manner, the second constraint reading $\eta + \kappa + \gamma' \ll d$, amounts to requiring $\eta \ll d$ which is also already imposed on $\eta$. Finally, for the third constraint appearing in~\eqref{eq:legal-2}, appeal again to the ability to choose $\kappa,\gamma,\gamma' \ll \eta$, so that the third constraint amounts to requiring 
$\eta \ll \frac{\alpha +d-1}{3(d+1)+2}$ which is also already imposed on $\eta$. 

We conclude our selection of constants by setting 
\begin{align}
0<\rho < \min\bigg\{\eta/4,\rho_{\absrefone}&(d,\alpha,\eta),\;2^{-1}\eta(1-\kappa-2\nu)^3\cdot\rho_{\pcref}(d,\zeta),\nonumber \\ 
&\rho_{\conrefone}(d,\alpha-\kappa-\gamma, \eta + \kappa + \gamma') \cdot (\nu(1-\kappa)/4)^3\bigg\}. \label{eq:rho-end}
\end{align}

Let $n$ be sufficiently large and let $H$ be an $n$-vertex $(\rho,d)_{\cherry}$-dense $3$-graph satisfying $\delta(H) \geq \alpha \binom{n-1}{2}$.
As $\rho < \rho_{\absrefone}(d,\alpha,\eta)$, $H$ admits a $(d-\eta,\mu,\kappa)$-absorbing-path $A$, by Lemma~\ref{lem:path-absorb-2}. Define $H' := H - V(A)$. 

We prepare for an invocation of Lemma~\ref{lem:res-2}. Let $D_{d-\eta}(H)$ denote the set of pairs $\{x,y\} \in \binom{V(H)}{2}$ satisfying $\deg_H(x,y) \geq (d-\eta)n$. Owing to $\eta \ll d$ and $\rho < \eta/4$, by~\eqref{eq:rho-end}, the set $D_{d-\eta}(H)$ has order of magnitude $\Omega(n^2)$, by~\eqref{eq:Bbeta-size-2}. For each pair $\{x,y\} \in D_{d-\eta}(H)$, set $U_{\{x,y\}} : = N_H(x,y) \cap V(H')$. For each vertex $x \in V(H)$, let $L_x := L_x(H)$ denote the link graph\footnote{See~\S~\ref{sec:cnt} for a definition.} of $x$ in $H$. Set $L_{x,H'}$ to be the subgraph of $L_x$ induced by $V(H')$. For any pair $\{x,y\} \in D_{d-\eta}(H)$, we may write 
$$
\deg_{H'}(x,y) = |U_{\{x,y\}}| \geq \deg_H(x,y) - |V(A)| \geq (d-\eta-\kappa)n. 
$$ 
For a vertex $x \in V(H)$, we may write
\begin{align*}
\deg_{H'}(x) & = e(L_{x,H'}) \geq e(L_x) - |V(A)|\cdot n \\
& \geq \alpha \binom{n-1}{2} - \kappa n^2\\
& \geq (\alpha - \kappa)\frac{n^2}{2} \quad \text{(as seen in~\eqref{eq:been-here})}\\
& \geq (\alpha -\kappa)\binom{n'-1}{2},
\end{align*}
where $n' := |V(H')|$. 

Lemma~\ref{lem:res-2}, applied with $\nu_{\resrefone} = \nu$, then asserts that there exists a set $R \subseteq V(H')$ of size $|R| = \nu n' \pm \nu (n')^{2/3}$, and thus satisfying~\eqref{eq:R-size}, such that 
\begin{equation}\label{eq:codeg-high}
|N_H(x,y) \cap R| \geq \big(d-\eta-\kappa - 2n^{-1/3}\big)|R|,
\end{equation}
whenever $\{x,y\} \in D_{d-\eta}(H)$, and such that for every $x \in V(H)$
$$
e(L_{x,H'}[R]) \geq \big(\alpha -\kappa -3n^{-1/3}\big)\binom{|R|}{2} 
$$
holds. In particular, 
\begin{equation}\label{eq:R-1-deg}
\delta(H[R]) \geq \big(\alpha -\kappa -3n^{-1/3}\big)\binom{|R|}{2}.
\end{equation} 

With $A$ and $R$ set, define $H'' := H' -R$. Then, $n'':= |V(H'')|$ satisfies~\eqref{eq:n2}. This estimate, coupled with 
$$
\rho \overset{\eqref{eq:rho-end}}{<} 2^{-1}\eta(1-\kappa-2\nu)^3\cdot\rho_{\pcref}(d,\zeta),
$$
and put through the argument seen in Claim~\ref{clm:H_2-dense},
collectively imply that 
\begin{equation}\label{eq:H''}
\text{$H''$ is $(\xi_1,d)_{\cherry}$-dense for some $\xi_1 < 2^{-1} \eta \cdot \rho_{\pcref}(d,\zeta)$}. 
\end{equation}

Recall the definition put forth in~\S~\ref{sec:positive} of the spanning subgraph $H''_{d-\eta} \subseteq H''$. In the following claim, we essentially appeal to~\eqref{eq:Hbeta} asserting that $H''_{d-\eta}$ is $\cherry$-dense. Here, however, we require much less.  

\begin{claim}
$H''_{d-\eta}$ is $(\xi_2,d)$-dense, for some $\xi_2 < \rho_{\pcref}(d,\zeta)$.
\end{claim}

\begin{innerproof}
Fix $X \subseteq V(H''_{d-\eta})$. Then,
$$
e_{H''_{d-\eta}}(X)  \geq e_{H''}(X) - |B_{d-\eta}(H'')| \cdot n''.
$$
Writing $\eta = d - (d-\eta)$ and appealing to both~\eqref{eq:Bbeta-size-2} and~\eqref{eq:H''}, we may write
\begin{align*}
e_{H''_{d-\eta}}(X) & \overset{\phantom{\eta \leq 1}}{\geq} d\binom{|X|}{3} - \frac{\rho_{\pcref}(d,\zeta)\eta}{2} (n'')^3 - \frac{\rho_{\pcref}(d,\zeta)\cdot\big(d-(d-\eta)\big)}{2\cdot \big(d-(d-\eta)\big)}(n'')^3\\
& \overset{\eta \leq 1}{\geq} d\binom{|X|}{3} - \rho_{\pcref}(d,\zeta)(n'')^3.
\end{align*}
As $H''_{d-\eta}$ spans $H''$, $|V(H''_{d-\eta})| = n''$ holds, and the claim follows. 
\end{innerproof}

Lemma~\ref{lem:path-cover} then asserts that $H''_{d-\eta}$ admits a collection $\P' := \{P_1,\ldots,P_{h-1}\}$, $h-1 \leq \ell_{\pcref}(d,\zeta)$, of vertex-disjoint paths covering all but at most $\zeta n'' \leq \zeta n$ vertices of $H''$ (recall that $H''_{d-\eta}$ spans $H''$) and thus of $H$ as well. Write $P_h := A$ and set $\P := \P' \cup \{P_h\}$. By definition of $A$ and $H''_{d-\eta}$,
\begin{equation}\label{eq:inD}
\{x,y\} \in D_{d-\eta}(H), \; \text{whenever $\{x,y\}$ is an end pair of some path $P \in \P$.}
\end{equation}

As in the proof of Theorem~\ref{thm:main}, we seek to construct a path $L$ of the form~\eqref{eq:L-path} and then close the latter into a (tight) cycle $C$. 
With the approach for this construction here conceptually identical to that seen in the proof of Theorem~\ref{thm:main}, we focus on the differences. More specifically, we are to show that $h$ applications of the connecting lemma relevant to the setting at hand, namely Lemma~\ref{lem:con-1deg}, can be carried out as to construct the cycle $C$. To that end, for $ i \in [h-1]$, define $L_i$ and $R_i$ as in~\eqref{eq:Li-partial} and~\eqref{eq:Ri-partial}, respectively. 
The following series of claims verifies that $H[R_i]$ adheres to the premise of Lemma~\ref{lem:con-1deg} so that the pairs $\{a,b\}$ and $\{c,d\}$ (per the definition of $R_i$) can be connected through $H[R_i]$. 

\begin{claim}\label{clm:Ri-min-deg}
$\delta(H[R_i]) \geq \big(\alpha -\kappa -\gamma \big)\binom{|R|}{2}$. 
\end{claim}

\begin{innerproof}
Observe that
$$
\delta(H[R_i]) \overset{\eqref{eq:R-1-deg}}{\geq} \delta(H[R]) - 10h \cdot n \geq \big(\alpha -\kappa -3n^{-1/3}\big)\binom{|R|}{2} - 10h \cdot n, 
$$
where the term $10h \cdot n$ accounts for all pairs involving a vertex in $V(L_i) \cap R$. 
As $10h\cdot n = O(n)$ (recall that $h-1 \leq \ell_{\pcref}(d,\zeta)$) and $|R| = \Omega(n)$, by~\eqref{eq:R-size}, the claim follows. 
\end{innerproof}

Next, we consider the codegree of the pairs $\{a,b\}$ and $\{c,d\}$, per the definition of $R_i$. 

\begin{claim}\label{clm:Ri-pair}
$\deg_{H[R_i]}(a,b), \deg_{H[R_i]}(a,b) \geq \big(d-\eta-\kappa - \gamma'\big)|R|$
\end{claim}

\begin{innerproof}
By~\eqref{eq:inD}, the pairs $\{a,b\},\{c,d\}$ lie in $D_{d-\eta}(H)$. Consequently~\eqref{eq:codeg-high} holds for both these pairs. 
Then, 
$$
\deg_{H[R_i]}(a,b) \geq \big(d-\eta-\kappa - 2n^{-1/3}\big)|R| - 10h. 
$$
As $10h = O(1)$ and $|R| = \Omega(n)$, by~\eqref{eq:R-size}, the claim follows for $\{a,b\}$. A similar argument holds for $\{c,d\}$.  
\end{innerproof}

Next, we address the $\cherry$-denseness of $H[R_i]$. 

\begin{claim}\label{clm:Ri-dense}
$H[R_i]$ is $(\xi_3,d)_{\cherry}$-dense, for some $\xi_3 < \rho_{\conrefone}(d,\alpha-\kappa-\gamma, \eta + \kappa + \gamma')$. 
\end{claim}

\begin{innerproof}
The set $R_i$ satisfies~\eqref{eq:Ri-size}. Then, owing to  
$$
\rho < \rho_{\conrefone}(d,\alpha-\kappa-\gamma, \eta + \kappa + \gamma') \cdot (\nu(1-\kappa)/4)^3,
$$
an argument identical to that seen in Claim~\ref{clm:H_2-dense} establishes the claim. 
\end{innerproof}

Owing to~\eqref{eq:legal-1} and~\eqref{eq:legal-2}, $\alpha -\kappa -\gamma + d - 1 >0$ and $\eta+ \kappa + \gamma' < g(\alpha - \kappa - \gamma,d)$ hold, respectively. These inequalities togehter with Claims~\ref{clm:Ri-min-deg},~\ref{clm:Ri-pair}, and~\ref{clm:Ri-dense} collectively assert that $H[R_i]$ and the pairs $\{a,b\}$ and $\{c,d\}$ satisfy the premise of Lemma~\ref{lem:con-1deg} with $\alpha_{\conrefone} = \alpha -\kappa -\gamma$, $d_{\conrefone} = d$, and $\eta_{\conrefone} = \eta + \kappa+\gamma'$. The pairs $\{a,b\}$ and $\{c,d\}$ can then indeed be connected through $H[R_i]$ as alleged. 

This completes the definition of the cycle $C$ containing the absorbing-path $A$. To prove that $C$ can be extended as to absorb all possibly uncovered vertices, use the argument seen for this in the proof of Theorem~\ref{thm:main}. This concludes our proof of Theorem~\ref{thm:main2}.

\vspace{3ex}
\noindent
{\large \bscaps{Acknowledgements.}} We would like to thank Yury Person and Mathias Schacht for their helpful remarks on an earlier draft of the manuscript. We are greatly indebted to two anonymous referees for their meticulous dedication towards improving this manuscript, for making numerous suggestions, finding mistakes, and overall treating this manuscript as their own.

\bibliographystyle{amsplain}
\bibliography{lit}
\end{document}